\documentclass[12pt]{article}

\usepackage[margin=1in]{geometry}
\usepackage{amsmath,amssymb,amsfonts,amsthm,mathrsfs,bbm}
\usepackage{indentfirst}
\usepackage{graphicx}

\usepackage{array}          %

\usepackage{xcolor}
\definecolor{lightblue}{rgb}{0,0.2,0.55}
\usepackage[colorlinks=true,linkcolor=lightblue,
            citecolor=lightblue,urlcolor=lightblue]{hyperref}
\usepackage{yfonts}
\usepackage[title]{appendix}
\usepackage{scalerel}
\usepackage{pgfplots}
\pgfplotsset{width=10cm,compat=1.9}
\usepackage{empheq}
\usepackage{float}
\usepackage{algorithm}
\usepackage{algpseudocode}

\usepackage[normalem]{ulem}

\usepackage{subcaption} 

\usepackage{enumerate} 

\usetikzlibrary{shapes.geometric,arrows.meta,positioning,calc} %

\makeatletter
\newcommand*\rel@kern[1]{\kern#1\dimexpr\macc@kerna}
\newcommand*\widebar[1]{%
  \begingroup
  \def\mathaccent##1##2{%
    \rel@kern{0.8}%
    \overline{\rel@kern{-0.8}\macc@nucleus\rel@kern{0.2}}%
    \rel@kern{-0.2}%
  }%
  \macc@depth\@ne
  \let\math@bgroup\@empty \let\math@egroup\macc@set@skewchar
  \mathsurround\z@ \frozen@everymath{\mathgroup\macc@group\relax}%
  \macc@set@skewchar\relax
  \let\mathaccentV\macc@nested@a
  \macc@nested@a\relax111{#1}%
  \endgroup
}
\makeatother

\makeatletter
\DeclareRobustCommand\widecheck[1]{{\mathpalette\@widecheck{#1}}}
\def\@widecheck#1#2{%
    \setbox\z@\hbox{\m@th$#1#2$}%
    \setbox\tw@\hbox{\m@th$#1%
       \widehat{%
          \vrule\@width\z@\@height\ht\z@
          \vrule\@height\z@\@width\wd\z@}$}%
    \dp\tw@-\ht\z@
    \@tempdima\ht\z@ \advance\@tempdima2\ht\tw@ \divide\@tempdima\thr@@
    \setbox\tw@\hbox{%
       \raise\@tempdima\hbox{\scalebox{1}[-1]{\lower\@tempdima\box
\tw@}}}%
    {\ooalign{\box\tw@ \cr \box\z@}}}
\makeatother

\allowdisplaybreaks           %

\renewcommand{\Im}{i} 
\newcommand{\real}{\mathbb{R}}
\newcommand{\R}{\mathbb{R}}
\newcommand{\C}{\mathbb{C}}
\newcommand{\N}{\mathbb{N}}
\newcommand{\E}{\mathbb{E}}
\newcommand{\PP}{\mathbb{P}}
\DeclareMathOperator*{\esssup}{ess\,sup}

\theoremstyle{plain}
\newtheorem{theorem}{Theorem}[section]
\newtheorem{proposition}[theorem]{Proposition}

\newtheorem{lemma}[theorem]{Lemma}

\theoremstyle{definition}
\newtheorem{definition}[theorem]{Definition}
\newtheorem{assumption}{Assumption}

\theoremstyle{remark}

\numberwithin{equation}{section}

\title{\huge Probabilistic representation and classical solutions of
   wave equations with complex polynomial nonlinearities }

\author{
  Joshua J.~Y.~Chan\thanks{School of Physical and Mathematical Sciences, Nanyang Technological University, Singapore. \texttt{joshuaju001@e.ntu.edu.sg}}
  \and
  Nicolas Privault\thanks{School of Physical and Mathematical Sciences, Nanyang Technological University, Singapore. \texttt{nprivault@ntu.edu.sg}}
}
\date{\today}

\begin{document}
\maketitle

\vspace{-0.6cm}

\begin{abstract}
 We review the probabilistic representation of
 solutions of wave equations with polynomial nonlinearities
 in spatial dimensions $d=1,2,3$ using stochastic branching processes.
 Under regularity assumptions on the initial data, we derive
 conditions ensuring
 the integrability of the corresponding Monte Carlo estimator,
 and the existence and smoothness of mild and classical
 solutions.
 We also present numerical results and
 comparisons with grid-based algorithms
 for the solution of nonlinear wave equations. 
\end{abstract}


\baselineskip=0.72cm

\section{Introduction}
\noindent 
The telegraph equation with dissipation
\begin{empheq}[left=\empheqlbrace]{align} 
  & \partial_{tt}u(x,t)+ 2 \lambda \partial_t u(x,t) 
  - \partial_{xx}u(x,t)
  = 0, 
  & (x,t)\in \real\times(0,\infty), \nonumber %
  \\
  \nonumber
  & u(x,0) = \phi(x),  & x\in\real,                             %
  \\
\nonumber %
  & \partial_t u(x,0) = 0, & x\in\real,                
\end{empheq}
 introduced in \cite{goldstein1951random} and
\cite{kac1974stochastic}, models wave propagation with finite velocity
in various physical contexts, including
electromagnetic wave propagation in transmission lines,
neutron transport, and biological systems.
Its solution admits the probabilistic representation 
$$u(x,t)=
\frac{1}{2} \E \left[\phi \left(
  x + \int_0^t (-1)^{N_s} ds \right) \right]
+
\frac{1}{2} \E \left[\phi \left(
  x + \int_0^t (-1)^{N_s} ds \right) \right]
,
$$
 where $(N_s)_{s\geq 0}$ is a standard
 Poisson process with rate $\lambda$ and 
 $\int_0^t (-1)^{N_s} ds$ is the Goldstein--Kac telegraph process
 which describes the random motion of a particle moving
 on the real line at constant speed, 
 changing direction at the arrival times of a Poisson process
 with rate $\lambda>0$, see also Chapter~12 of \cite{cartier}. 
 
\medskip

 In \cite{dalang},
 a probabilistic representation for the
 solution of the semilinear wave equation 
\begin{empheq}[left=\empheqlbrace]{align} 
 & \partial_{tt}u(x,t)-\Delta u(x,t) = V(x,t)u(x,t)
  & (x,t)\in\real^d \times(0,\infty),
\nonumber %
  \\
  \nonumber
  & u(x,0) = \phi(x),  & x\in\real^d,
  \\
\nonumber %
  & \partial_t u(x,0) = \psi(x), & x\in\real^d,                
\end{empheq}
 was constructed
 using a Poisson-driven continuous-time Markov process.
 In the case where the nonlinearity $f(u(x,t))$ is given
 by a convergent power series,
 nonlinear wave equations of the form
\begin{empheq}[left=\empheqlbrace]{align} 
 & \partial_{tt}u(x,t)-\Delta u(x,t) = f(u(x,t))
  & (x,t)\in\real^d \times(0,\infty), \label{eq:WaveMain2}
  \\
  \nonumber
  & u(x,0) = \phi(x),  & x\in\real^d,
  \\
\nonumber %
  & \partial_t u(x,0) = \psi(x), & x\in\real^d,                
\end{empheq}
 were treated in \cite{bakhtin} using
 the stochastic cascade method \cite{sznitman,waymire,ramirez}. %

 \medskip

 On the other hand, stochastic branching methods 
 \cite{skorohodbranching,inw,hpmckean}
 have been applied to the representation of solutions
 of parabolic and elliptic partial differential
 equations in e.g.
 \cite{hpmckean},
 \cite{chakraborty}, 
 \cite{laborderespa},
 \cite{penent2022fully}.

 \medskip
 
 Existence of mild solutions of nonlinear wave equations
 of the form \eqref{eq:WaveMain2}
 with general power series nonlinearities
 has been derived in \cite{labordere2} 
 by the stochastic branching method, 
 together with their probabilistic representation,
 using inverse Fourier representations
 of Green functions, under integrability conditions
 based on ODE arguments. 
 Lifespan estimates
  for the existence times
  of solutions to wave equations
  have been obtained by analytic methods in e.g.
  \cite{AgemiTakamura1992},
  \cite{zhouyi1992}, 
  \cite{zhouyi1993}, 
  \cite{TakamuraWakasa2011}.

 \medskip
   
 In this paper, we consider the application of
 the stochastic branching method to 
 nonlinear wave equations of the form 
\begin{empheq}[left=\empheqlbrace]{align} 
 & \partial_{tt}u(z,t)-c^2\Delta u(z,t) = f (u(z,t)),
  & (z,t)\in\C^d\times(0,\infty), \label{eq:WaveMain}
  \\
  \nonumber
  & u(z,0) = \phi(z),  & z\in\C^d,                             %
  \\
  \label{eq:IC2}
  & \partial_t u(z,0) = \psi(z), & z\in\C^d,                
\end{empheq}
in dimensions $d = 1,2,3$, where
\(c\in\C\setminus\{0\}\) and $f$ is a complex polynomial
nonlinearity of the form 
\begin{equation}\label{eq:PolyNonlin}
  f(u)=\sum_{k=0}^N a_ku^k, %
\end{equation}
 where
 $(a_k)_{0\leq k \leq N} \subseteq \C^{N+1}$, $N\ge1$,
 $a_{N}\neq0$ and $\Delta$ is the Laplacian 
$$
 \Delta=\sum_{k=1}^d\partial_{z_{k}}\partial_{z_{k}}.
$$ 
 In this context, our aim is two-fold.
 \begin{enumerate}[1)]
 \item
   We revisit the application of the stochastic branching method
   to wave equations 
   with polynomial nonlinearities %
   using the classical integral kernel
   expressions of~\cite{Eva10}.
\item %
  In addition to mild solutions,
  we obtain sufficient conditions for the existence of classical solutions
  using branching process domination arguments. 
 \end{enumerate}
\noindent
 Given $(z_0,t_0)\in \C^d\times(0,\infty)$, consider
 the space-time cone 
 \begin{equation}
   \label{space-time-cone}
 \widebar{\Gamma}_{z_0,t_0}(\C^d) 
 :=\bigl\{z\in \C^d:\exists y\in\R^d\text{ such that }z=z_0+cy 
 \text{ and } |y|\le t_0
 \bigr\}, 
\end{equation}
 where $c$ is given in \eqref{eq:WaveMain}, 
 and we let 
 ${\cal C} \bigl(
   \widebar{\Gamma}_{z_0,t_0}(\C^d)
 \bigr)$ and 
${\cal C}^1 \bigl(
  \widebar{\Gamma}_{z_0,t_0}(\C^d)
  \bigr)$ 
  respectively denote
  the sets of continuous and continuously
  differentiable functions on
  $\widebar{\Gamma}_{z_0,t_0}(\C^d)$.

    \medskip
     
    In Theorems~\ref{thm:Mild} and \ref{thm:mild-existence}
 we obtain the existence of {mild solutions}
 of \eqref{eq:WaveMain}-\eqref{eq:PolyNonlin}
 on a certain time interval $[0,t_0]$
 by assuming that
 $\phi, \psi$
 are bounded measurable functions on
 $\widebar{\Gamma}_{z_0,t_0}(\C^d)$ 
 and 
  $\phi\in {\cal C}^1 \bigl(
  \widebar{\Gamma}_{z_0,t_0}(\C^d)
  \bigr)$ if $d=2,3$. 
\noindent 
 Then, in Theorem~\ref{thm:classical} we show that 
 \eqref{eq:WaveMain}-\eqref{eq:PolyNonlin}
 admits a {classical solution} on $[0,t_0]$ 
 by assuming in addition that
 $\psi\in {\cal C}^2 \bigl( \widebar{\Gamma}_{z_0,t_0}(\C^d) \bigr)$ and 
\begin{enumerate}[i)] %
\item $\phi \in {\cal C}^2 \bigl(
  \widebar{\Gamma}_{z_0,t_0}(\C^d)
  \bigr)$ for $d=1$, or  
\item $\phi\in {\cal C}^3 \bigl(
  \widebar{\Gamma}_{z_0,t_0}(\C^d)
  \bigr)$ for $d=2,3$. 
\end{enumerate}
 Our approach relies on 
 a direct use of the classical d'Alembert, Poisson and Kirchhoff kernels
 in dimensions $d=1,2,3$.
 This cone-based construction yields a branching process 
 whose spatial marks are localized on admissible cones.
 It provides local stochastic representations
 of mild solutions and, under additional regularity assumptions,
 of classical solutions.
 This complements the approach of \cite{labordere2}
 which relies on a general Fourier-based framework for constant-coefficient Cauchy
 problems, in which Green functions are constructed through inverse Fourier
 transforms, and integrability is controlled by an auxiliary ODE. 

 \medskip

 Section~\ref{s7} presents numerical experiments
 with comparisons to other numerical methods. 
   Our experiments show in particular that, 
   although the branching algorithm is explosive in nature,
   it can outperform standard (non optimized or customized)
   grid-based codes in terms of
   stability beyond a given time horizon, 
   including for an ill-posed focusing equation. 

   \medskip
    
We proceed as follows.
Section~\ref{s2} gathers preliminaries on kernels and the weak and
mild formulations of \eqref{eq:WaveMain}-\eqref{eq:IC2}. 
Section~\ref{s3} presents the random trees
used for probabilistic representations.
In Section~\ref{s4}, 
we provide sufficient conditions for the
existence and probabilistic representation of mild solutions,
under integrability conditions on the weights used in
stochastic representations.
 Section~\ref{s5} derives conditions on equation coefficients under 
which this integrability holds. 
 In Section~\ref{s6}, we provide
conditions for the existence of classical solutions. 
Numerical simulations and comparisons with grid-based
finite differences and other 
methods are presented for wave and elliptic problems
in Section~\ref{s7}.

\section{Preliminaries}
\label{s2}
\subsection{Kernel formulation}
\label{subsec:LinearTheory}
\noindent
 This subsection reviews basic results
 for Cauchy problems of the form 
\begin{empheq}[left=\empheqlbrace]{align}
 & \partial_{tt}u(x,t)-c^2\Delta u(x,t) = g(x,t), 
  &  (x,t)\in\R^d\times(0,\infty),   \nonumber %
  \\
  & u(x,0)=\phi(x),                       & x\in\R^d,   \nonumber %
  \\
 & \partial_t u(x,0)=\psi(x),           & x\in\R^d. \nonumber %
\end{empheq}
 Duhamel's principle yields the mild form
\begin{equation}
\nonumber %
  u(x,t)
  =
  \partial_t  \int_{\R^d}   \phi(x-cy)K(dy,t) 
   + 
  \int_{\R^d}  \psi(x-cy)K(dy,t) 
   + 
  \int_0^t \int_{\R^d} g(x-cy,t-s)K(dy,s) ds,
\end{equation}
 of \eqref{eq:WaveMain}-\eqref{eq:IC2}, where $K(dy,t)$
 is the kernel defined as 
\begin{equation}
   \nonumber %
 K(dy,t)=
 \begin{cases}
    \displaystyle
    \frac{1}{2}\mathbbm1_{\{|y|<t\}}dy, &
    d=1,\\[8pt]
 \nonumber
        \displaystyle
    \frac{1}{2\pi} \mathbbm1_{\{|y|<t\}} \frac{dy}{\sqrt{t^2-|y|^2}}, &
    d=2,\\[8pt]
\nonumber
         \displaystyle
    \frac{1}{4\pi t} \sigma_t^{(2)}(dy), &
    d=3,
  \end{cases}
\end{equation}
 where $|\cdot |$ denotes Euclidean norm  
 and $\sigma_t^{(2)}$ is the surface measure on the $2$-dimensional
 sphere $S_2 (0,t)$. 
 From \cite[§2.4]{Eva10}, we can distinguish the following
 cases,
 where ${\cal C}^k(\real^d)$
 denotes the set of functions on $\real^d$ that are
 continuously differentiable of order $k\geq 0$,
 with ${\cal C}(\real^d) = {\cal C}^0(\real^d)$.
\begin{enumerate}[i)] %
\item
\noindent
 $d=1$. Under the conditions
 $\phi \in {\cal C}^2(\R)$, 
 $\psi\in {\cal C}^1(\R)$ and 
 $g\in {\cal C}^1 (\R\times[0,\infty))$,
 from \cite[§2.4,~Eqs.~(8),~(43)]{Eva10}, 
 $u(x,t)$ satisfies d'Alembert's formula 
\begin{equation}
\label{eq:Sol1D}
 u(x,t)=\frac{1}{2}\bigl(
 \phi(x+ct)+\phi(x-ct)\bigr) 
         +\frac{1}{2} \int_{-t}^{t} \psi(x+cy) dy
         +\frac{1}{2} \int_0^t \int_{-s}^{s} g(x+cy,t-s) dy ds, 
\end{equation} 
 $(x,t)\in \real\times \real_+$. 
\item $d$ even, $d\geq 2$.
  Under the conditions $\phi\in {\cal C}^{2+d/2}(\R^d)$,
$\psi\in {\cal C}^{1+d/2}(\R^d)$ and 
$g\in {\cal C}^{1+d/2}\bigl(\R^d\times[0,\infty)\bigr)$,
for $d=2m$, $m\in \N$, %
  from \cite[§2.4,~Eqs.~(38),~(41)-(42)]{Eva10} we have 
\begin{align}
     \nonumber
  u(x,t) & =
  \frac{1}{(2\pi)^m}
  \partial_t \left(\frac{1}{t}\partial_t \right)^{m-1}
        \int_{B_d(0,t)} \frac{\phi(x+cy)}{\sqrt{t^2-|y|^2}} dy
        +\frac{1}{(2\pi)^m}
        \left(\frac{1}{t}\partial_t \right)^{m-1}
     \int_{B_d(0,t)} \frac{\psi(x+cy)}{\sqrt{t^2-|y|^2}} dy
          \\[4pt]
\nonumber %
     &
  \quad   +\frac{1}{(2\pi)^m}
     \int_0^t
     \left(\frac{1}{s}\partial_{s}\right)^{m-1}
         \int_{B_d(0,s)} \frac{g(x+cy,t-s)}{\sqrt{s^2-|y|^2}} dy ds
, 
\end{align}
$(x,t)\in \real^d\times (0,\infty )$,
 where $B_d(0,t)$ denotes the
 $d$-dimensional open ball with radius $t\geq 0$. 
 In particular, for $d=2$ this yields the Poisson planar kernel 
\begin{align}
  \label{eq:PlanarKernel}
    u(x,t) & = \frac{1}{2\pi t}
 \int_{B_2(0,t)}
\frac{ \phi(x+cy)+ \nabla \phi(x+cy) \cdot (cy)+ t \psi(x+cy)}
             {\sqrt{t^2-|y|^2}} dy
             \\
             \nonumber
             & \quad + 
\frac{1}{2\pi}\int_0^t  \int_{B_2(0,s)}
        \frac{g(x+cy,t-s)}{\sqrt{s^2-|y|^2}} dy ds, 
\end{align} 
 $(x,t)\in \real^2 \times (0,\infty )$, where 
$$
 \nabla \phi(x) \cdot  y:=\sum_{k=1}^dy_{k}\partial_{x_{k}}\phi(x),
 \quad
  x,y \in \real^d, 
$$
 see \cite[§2.4,~Eq.~(27)]{Eva10}. 
\item $d$ odd, $d\geq 3$. 
 Under the conditions $\phi\in {\cal C}^{(d+3)/2}(\R^d)$,
 $\psi\in {\cal C}^{(d+1)/2}(\R^d)$ and 
 $g\in {\cal C}^{(d+1)/2} (\R^d\times[0,\infty))$,
   for $d=2m+1$, $m\in \N$,
   from \cite[§2.4,~Eqs.~(31),~(41)-(42)]{Eva10} we have 
\begin{align}
\nonumber %
  u(x,t) & =
    \frac{1}{2(2\pi)^m}
\partial_t \left(\frac{1}{t} \partial_t \right)^{m-1}
\left(
\frac{1}{t} \int_{S_{d-1} (0,t)} \phi(x+cy) \sigma^{(d-1)}_t (dy)
\right)
\\[4pt]
   \nonumber
   & \quad +\frac{1}{2(2\pi)^m}\left(
   \frac{1}{t}\partial_t \right)^{m-1}
   \left(
   \frac{1}{t} \int_{S_{d-1} (0,t)} \psi(x+cy) \sigma^{(d-1)}_t (dy)
   \right)
   \\[4pt]
   \nonumber
          &
   \quad +\frac{1}{2(2\pi)^m}\int_0^t
   \left(\frac{1}{s}\partial_{s}\right)^{m-1}
        \left(
        \frac{1}{s}
        \int_{S_{d-1} (0,s)}
        g(x+cy,t-s) \sigma^{(d-1)}_s (dy)
        \right)
        ds
,
\end{align}
\noindent
 $(x,t)\in \real^d \times (0,\infty )$,
 where $\sigma_s^{(d-1)}$ is the {surface measure} on the
 $(d-1)$-dimensional sphere $S_{d-1} (0,s)$. 
 In particular, for $d=3$ we find the Kirchhoff spatial kernel 
\begin{align}
\nonumber
\displaystyle
      u(x,t)& = 
      \frac{1}{4\pi t^2}
        \int_{S_2 (0,t)}
        \bigl(
        \phi(x+cy)+\nabla \phi(x+cy) \cdot (cy)+t\psi(x+cy)\bigr) 
         \sigma^{(2)}_t (dy) \\[6pt]
   \label{eq:SpatialKernel}
         & \quad +\frac{1}{4\pi}
        \int_{B_3(0,t)}
        g(x+cy,t-|y|)
        \frac{dy}{|y|}, 
\end{align}
 $(x,t)\in \real^3 \times (0,\infty )$,
 see \cite[§2.4,~Eqs.~(22),~(44)]{Eva10}.
\end{enumerate} 

\subsection{Weak formulation} %
\noindent
Let $U\subset\R^d$ be an open set and $k\geq 0$,
 $p\in [1,\infty ]$. 
 We consider the Sobolev space
\[
  H^{k,p}(U)
  :=\bigl\{u\in L^1_{\mathrm{loc}}(U) \ \! : \ \! \nabla^\alpha u\in L^p (U)\text{ for all }|\alpha|\le k\bigr\} 
\]
and its local version 
$$ 
H^{k,p}_\text{loc}(U):=\{
u\in L^1_{\mathrm{loc}}(U)
 \ \! : \ \! \nabla^\alpha u\in L^p_\text{loc}(U)\text{ for }|\alpha|\le k\}
$$ 
 which allows for growth at infinity.  For any
  Banach space $X$, %
   we also set
\[
L^{\infty}((0,T);X):=\bigl\{v\ :\ (0,T)\to X \ : \
\esssup_{t\in (0,T)}\|v(t)\|_{X}<\infty\bigr\},
 \quad T>0.
\]
 The next two results
 justify our later branching scheme.
\begin{proposition}
\label{thm:GlobalLipschitz}
\textit{(Global existence for Lipschitz $f$, \cite[\S 12.2.1 Theorem 1]{Eva10}).}
For $T>0$
let $f:\R^d\times\R\times\R\longrightarrow\R$ 
be globally Lipschitz, and let 
$\phi \in H^{1,2}_\text{loc}(\R^d)$
and 
$\psi \in L^2_\text{loc}(\R^d)$.
 Then, the initial-value problem 
\begin{empheq}[left=\empheqlbrace]{align}
 & u_{tt} (x,t) -\Delta u(x,t)+f(\nabla u (x,t),\partial_tu(x,t),u(x,t))=0,
  && (x,t)\in\R^d\times (0,T],  \label{eq:SLwavePDE}
    \\
 & u(x,0)=\phi,
    && x \in \R^d , \nonumber %
    \\
 & \partial_tu (x,0)=\psi,
      && x \in \R^d.  \label{eq:SLwaveIC2}
\end{empheq}
   admits a unique weak solution. 
   In addition, if
   $\phi \in H^{2,2}_\text{loc} (\R^d)$
   and 
   $\psi \in H^{1,2}_\text{loc}(\R^d)$,
   we have 
\[
u\in L^{\infty} \bigl((0,T);H^{2,2}_\text{loc}(\R^d)\bigr),
 \ \ 
  \partial_t u\in L^{\infty} \bigl((0,T);H^{1,2}_\text{loc}(\R^d)\bigr), 
 \mbox{ and } 
  \partial_{tt}u\in L^{\infty} \bigl((0,T);L^2_\text{loc}(\R^d)\bigr).
\]
\end{proposition}
\noindent 
When $f$ is a polynomial of degree $N\ge2$ (smooth but not globally Lipschitz),
the hypotheses of Proposition~\ref{thm:GlobalLipschitz} fail; in this case, 
Proposition~\ref{thm:LocalSmooth} guarantees a solution on a possibly small
time interval.
\begin{proposition}
  \label{thm:LocalSmooth}
  \textit{(Short-time existence for smooth $f$,
    \cite[\S 12.2.2 Theorem 3]{Eva10}).}
  Assume that $f\in {\cal C}^{\infty}$, and let $k>1+d/2$.
  Then, for any initial data
  $\phi \in H^{k,2}(\R^d)$
  and 
  $\psi \in H^{k-1,2}(\R^d)$
  there exists $T>0$ such that
 the solution $u$ of 
 \eqref{eq:SLwavePDE}-\eqref{eq:SLwaveIC2}
 satisfies 
$$
 u\in L^{\infty} \bigl(0,T;H^{k,2}(\R^d)\bigr) \ \ \mbox{and }
 \ \ \partial_t u\in L^{\infty} \bigl(0,T;H^{k-1,2}(\R^d)\bigr).
 $$ 
\end{proposition}
\noindent 
It is further shown in \cite[§12.5.2]{Eva10} that, for the prototype
$f(u) = - |u|^p$ in dimension $d=3$, no global smooth solution exists when $1<p<1+\sqrt{2}$, even for arbitrarily small, compactly supported data.
This non-existence justifies our focus on
 short-time 
Monte Carlo approximations in the higher order polynomial regime.

\subsection{Mild formulation} 
\label{subsec:MildSol}

\noindent
In what follows, we state a mild formulation for the complex 
 space problem \eqref{eq:WaveMain}-\eqref{eq:PolyNonlin}.

\begin{definition}
A function \(u:\C^d\times[0,\infty)\to\C\) is a \emph{mild solution}
of \eqref{eq:WaveMain}-\eqref{eq:PolyNonlin} if it satisfies the following integral identities.
\begin{enumerate}[i)] %

\item
\noindent
\(d=1\). In this case, using d'Alembert's formula~\eqref{eq:Sol1D},
 $u(z,t)$ must satisfy
\begin{equation}\label{eq:Mild1D}
  u(z,t)=\frac{1}{2}(\phi(z+ct)+\phi(z-ct))
         +\frac{1}{2}\int_{-t}^{t}\psi(z+cy)\,dy
         +\frac{1}{2}\int_0^t\int_{-s}^{s}
           f\bigl(u(z+cy,t-s)\bigr)\,dy\,ds,
\end{equation}
for \((z,t)\in\C\times[0,\infty)\).

\item
\noindent
\(d=2\). In this case, using the planar Poisson kernel~\eqref{eq:PlanarKernel},
 $u(z,t)$ must satisfy
\begin{align}
\label{qakjfds-1} 
  u(z,t)
  &= \frac{1}{2\pi}
     \int_{B_2(0,t)}
     \frac{\phi(z+cy)+\nabla \phi(z+cy)\cdot(cy)+t\,\psi(z+cy)}
          {t\sqrt{t^2-|y|^2}}\,dy
\\[-2pt]
  \nonumber
  &\quad
  + \frac{1}{2\pi}\int_0^t\int_{B_2(0,s)}
     \frac{f\bigl(u(z+cy,t-s)\bigr)}{\sqrt{s^2-|y|^2}}\,dy\,ds,
  \quad (z,t)\in \C^2 \times(0,\infty),
\end{align}
 provided that $\phi\in {\cal C}^1 (\C^d)$,
 i.e., in polar form, letting
 $\hat{e}_\theta :=(\cos\theta,\sin\theta)$, $\theta\in[0,2\pi]$,
 we have 
\begin{align}
  \nonumber
  u(z,t)
  &= \frac{1}{2\pi} \int_0^t\int_0^{2\pi}
     \frac{s}{t\sqrt{t^2-s^2}}
     \bigl(\phi(z+cs\hat{e}_\theta )
           +\nabla \phi(z+cs\hat{e}_\theta )\cdot(cs\hat{e}_\theta )
           +t\,\psi(z+cs\hat{e}_\theta )\bigr)\,d\theta\,ds
  \\[6pt]
  \label{eq:Mild2Dpolar}
  &\quad
  + \frac{1}{2\pi} \int_0^t\int_0^{s}\int_0^{2\pi}
    \frac{r}{\sqrt{s^2-r^2}}
    \,f\bigl(u(z+cr\hat{e}_\theta ,t-s)\bigr)\,d\theta\,dr\,ds,
  \quad (z,t)\in \C \times(0,\infty). 
\end{align}
\item
\noindent
\(d=3\). In this case, using the spatial Kirchhoff kernel~\eqref{eq:SpatialKernel},
 $u(z,t)$ must satisfy
\begin{align}
\label{qakjfds-2} 
  u(z,t)
  &= \frac{1}{4\pi t^2}
     \int_{S_2(0,t)}
     \bigl(\phi(z+cy)+\nabla \phi(z+cy)\cdot(cy)+t\,\psi(z+cy)\bigr)\,
     \sigma^{(2)}_t (dy)
\\
  \nonumber
  &\quad
  + \frac{1}{4\pi}\int_{B_3(0,t)}
    f\bigl(u(z+cy,t-|y|)\bigr)\,\frac{dy}{|y|},
  \quad (z,t)\in \C^3 \times(0,\infty),
\end{align}
 provided that $\phi\in {\cal C}^1 (\C^d)$,
 or, in spherical coordinates,
 letting $\hat{e}_{\alpha,\theta}:=(\sin\alpha\cos\theta,\sin\alpha\sin\theta,\cos\alpha)$,
 $\theta\in[0,2\pi]$, $\alpha\in[0,\pi]$, we have 
\begin{align}
  \nonumber
  u(z,t)
  &= \frac{1}{4\pi}\int_0^\pi \int_0^{2\pi}
     \bigl(\phi(z+ct\hat{e}_{\alpha,\theta})
           +\nabla \phi(z+ct\hat{e}_{\alpha,\theta})\cdot(ct\hat{e}_{\alpha,\theta})
           +t\,\psi(z+ct\hat{e}_{\alpha,\theta})\bigr)
     \,d\theta\,\sin\alpha\,d\alpha
\\[6pt]
  \label{eq:Mild3Dpolar}
  &\quad
  + \frac{1}{4\pi}\int_0^{t}\int_0^\pi \int_0^{2\pi}
    s\,f\bigl(u(z+cs\hat{e}_{\alpha,\theta},t-s)\bigr)
    \,d\theta\,\sin\alpha\,d\alpha\,ds,
  \quad (z,t)\in \C^3 \times(0,\infty). 
\end{align}
\end{enumerate}
\end{definition}
\noindent
 For example,
   when $d=1$, 
   $\phi (z) = \mathbbm{1}_{[0,\infty )}(
     \operatorname{Re}(z)
     )$, 
   $\psi (z)= 0$ and $f(x)=0$, \eqref{eq:Mild1D}
     yields the mild solution
     $$
     u(z,t) =
     \frac{1}{2}
     \mathbbm{1}_{[-t,t)}(
       \operatorname{Re}(z)
       )
       +
       \mathbbm{1}_{[t,\infty )}(
         \operatorname{Re}(z)
         ),
         \quad
         (z,t)\in\C\times[0,\infty).
           $$

\section{Random trees} %
\label{s3}
\noindent 
\noindent
We now consider a stochastic branching
functional whose Monte Carlo expectation reproduces the mild
formulations
of the previous section.
For this, we construct 
a {marked Galton--Watson tree}
 on a probability space $(\Omega,\mathcal F,\PP )$, 
 which encodes the polynomial nonlinearity of \eqref{eq:PolyNonlin}, and whose marks reproduce the linear wave propagation.
 We will also construct a random
   functional $\mathcal{H}(z,t)$ whose expectation will be shown
   in Section~\ref{s4} to coincide with the value
   $u(z,t)$ of the
   mild solution introduced in \S\ref{subsec:MildSol}.  The ingredients are:  
\begin{itemize}
\item an {offspring mechanism} reflecting the coefficients $(a_k)_{k=0}^N$;  
\item {exponential lifetimes} modelling the time variable; and  
\item {spatial marks} $X_\kappa $ that emulate the light cone spread of the $d$-dimensional wave kernel.
\end{itemize}
\noindent
Table~\ref{jdklsda1} records the correspondence between the branching construction and the terms in the mild formulation of the PDE:
\begin{table}[H] 
\centering
\renewcommand{\arraystretch}{1.25}
\begin{tabular}{>{\raggedright\arraybackslash}p{0.32\textwidth}>{\raggedright\arraybackslash}p{0.58\textwidth}}
\hline
\textbf{Branching object} & \textbf{Role in the PDE representation} \\
\hline
Exponential lifetime $\tau_\kappa$ & Randomizes the time variable in the Duhamel integral. \\
Offspring number $J_\kappa$ & Encodes the polynomial nonlinearity $f(u)=\sum_{j=0}^N a_j u^j$, with $J_\kappa=j$ corresponding to the $j$-th power of the solution. \\
Spatial mark $X_\kappa$ & Samples the classical wave kernel and keeps the spatial displacement inside the light cone. \\
Boundary node $\kappa' \in {\cal K}^{\rm b}(t,\kappa )$ & Contributes the initial data terms involving $\phi$, $\psi$, and, in dimensions $d=2,3$, $\nabla\phi$. \\
Interior node $\kappa' \in {\cal K}^\circ (t,\kappa )$ & Contributes the nonlinear branching weight $a_{J_{\kappa'}}/q_{J_{\kappa'}}$ together with the time factor from the Duhamel integral. \\
Product over descendants & Represents the recursive multiplication of independent subtrees, corresponding to powers of $u$ in the polynomial nonlinearity. \\
\hline
\end{tabular}
\caption{Branching tree notation.}
\label{jdklsda1}
\end{table}
\noindent
We describe these elements in turn.

\subsection{Tree structure}
\noindent Particles are denoted by multi-indices
 $\kappa=(1,\kappa_2,\dots,\kappa_{|\kappa|})\in\N^{|\kappa|}$
 in the set of sequences 
$$
 \mathbb{K} := \{ (1) \} \cup \bigcup_{n\geq 1}
 \big(
 \{(1) \} \times \{1,\ldots , N\}^n
 \big)
$$ 
 where $(1)$ denotes the root particle.
\smallskip
\begin{enumerate}[i)] %
    \item Each particle $\kappa$ lives an exponential time $\tau_\kappa \sim\mathrm{Exp}(\lambda)$.
    \item Upon death, a particle
      $\kappa=(1,\kappa_2,\dots,\kappa_{|\kappa|})\in\N^{|\kappa|}$
      produces \(J_\kappa\) offspring,
      labeled
      $$
      (\kappa , j)
      =(1,\kappa_2,\dots,\kappa_{|\kappa|},j),
      \quad j=1,\ldots ,J_\kappa .
      $$
      The offspring count \(J_\kappa\) is an integer-valued random variable
      supported on \(\{0,\dots,N\}\) with
      \[
        \PP (J_\kappa=j)=q_j,\qquad j=0,1,\dots,N,
      \]
      where \(q_j=0\) if and only if \(a_j=0\).
      (If \(J_\kappa=0\), no offspring are produced.)
\end{enumerate}
\noindent
In addition, the random variables $\bigl\{\tau_\kappa ,J_\kappa \bigr\}_\kappa $ are i.i.d.\ and independent of the spatial marks introduced below,
and we write $\kappa\preceq\kappa'$ when $\kappa$ is an ancestor of $\kappa'$.

\subsection{Spatial marks}
\noindent
 For each $\kappa \in \mathbb{K}$ we let 
 $$
 Y_\kappa := (2U_\kappa -1)
 \quad
 \mbox{and}
 \quad 
 R_\kappa := \sqrt{1-(1-U_\kappa )^2},
 \quad
 \alpha_\kappa :=\arccos(1-2U_\kappa ),
 $$
 where $(U_\kappa )_{\kappa \in \mathbb{K}}$ is a sequence of
 i.i.d. uniformly distributed random variables on $[0,1]$.
 We consider the $d$-dimensional displacements $X_\kappa :[0,\infty)\to\R^d$
 defined by
 \begin{equation}
   \nonumber %
  X_\kappa (s) :=
  \begin{cases}
    s Y_\kappa , & d=1,\\[4pt]
\nonumber
          \bigl(
        s R_\kappa \cos\Theta_\kappa ,
        s R_\kappa \sin\Theta_\kappa 
      \bigr), & d=2,\\[4pt]
\nonumber
          \bigl(
        s \sin\alpha_\kappa \cos\Theta_\kappa ,
        s \sin\alpha_\kappa \sin\Theta_\kappa ,
        s \cos\alpha_\kappa 
      \bigr), & d=3, 
  \end{cases}
\end{equation}
 $s\geq 0$,
 where $(\Theta_\kappa )_{\kappa \in \mathbb{K}}$
 is a sequence of i.i.d. uniformly distributed
 random variables on $[0,2\pi]$. %
 We note that for \(s>0\), 
  \begin{align}
    \begin{cases}
      d=1:&
      ~sY_\kappa~\mbox{has density~} 
  y \longmapsto \displaystyle\frac{1}{2s} \mathbf1_{[-s,s]}(y),   
     \\[10pt]
\nonumber
d=2:&
 ~(s R_\kappa ,\Theta_\kappa)~\mbox{has density~} 
  (r,\theta) \longmapsto 
    \displaystyle\frac{r}{2\pi s\sqrt{s^2-r^2}} 
    \mathbf1_{[0,s)}(r) \mathbf1_{[0,2\pi)}(\theta); 
    \\[10pt]
\nonumber
d=3:&
 ~(\alpha_\kappa,\Theta_\kappa)~\mbox{has density~} 
  (\alpha,\theta) \longmapsto 
    \displaystyle\frac{\sin\alpha}{4\pi} 
    \mathbf1_{[0,\pi]}(\alpha) \mathbf1_{[0,2\pi)}(\theta). 
\end{cases}
\end{align} 

\subsection{Ancestry and time-truncated tree}
\noindent Given two nodes $\kappa , \kappa' \in \mathbb{K}$, we let 
\begin{itemize}
\item 
    $          {\rm DS}(\kappa ) :=\{\kappa'' : \kappa \preceq\kappa'' \}$
              denote the descendants of
              $\kappa$, and
            \item
            $
          {\rm AN}(\kappa,\kappa' ) :=\{\kappa'' \ : \ \kappa \preceq \kappa'' \preceq\kappa' \}$ 
          denote the
          ancestors of $\kappa'$ between $\kappa$ and $\kappa'$.
\end{itemize}
Given $t\ge0$,
we also consider the random truncated tree ${\cal K}(t,\kappa)$
 that contains $\kappa$ and any descendant $\kappa'$ of $\kappa$
  that satisfies 
$$ 
    \sum_{\kappa''\in {\rm AN}(\kappa,\kappa')\setminus\{\kappa' \}}
    \tau_{\kappa''}<t.
$$ 
 In particular, we have $\kappa \in {\cal K}(0,\kappa)$. 
 The set of boundary nodes of ${\cal K}(t,\kappa)$ is defined as 
$$
    {\cal K}^{\rm b} (t,\kappa)
    :=\bigg\{\kappa' \in{\cal K}(t,\kappa):\textstyle
    \displaystyle
    \sum_{\kappa'' \in {\rm AN}(\kappa,\kappa' )}
    \tau_{\kappa''}\ge t\bigg\},
$$
and the set of its interior nodes is defined as 
$$
{\cal K}^\circ(t,\kappa)
:={\cal K}(t,\kappa)\setminus
  {\cal K}^{\rm b} (t,\kappa), 
$$ 
  with $\kappa \in {\cal K}^{\rm b} (t,\kappa)$
  if $0 \le t \le \tau_{\kappa}$. 
  In addition, we consider the truncated lifetimes
  at time $t>0$, defined as
  \begin{equation}
    \nonumber %
        T_{\kappa'}^{t,\kappa}
        :=\min\bigl\{\tau_{\kappa'} ,
         t-  \sum_{\kappa''\in {\rm AN}(\kappa,\kappa')\setminus\{\kappa'\}}\tau_{\kappa''}\bigr\}, 
      \end{equation}
  with
  \(T_{\kappa}^{t,\kappa}=\tau_{\kappa}\text{ if }\tau_{\kappa}<t\) and \(T_{\kappa}^{t,\kappa}=t\) otherwise. For any interior node $\kappa'\in{\cal K}^\circ(t,\kappa)$, we still have \(T_{\kappa'}^{t,\kappa}=\tau_{\kappa'} \).
\subsection{Cumulative displacement}
\noindent
 Let $\kappa \in \mathbb{K}$ and $t\geq 0$. 
\begin{enumerate}[i)] %
\item {$d=1$.} 
 For every $\kappa'\in{\cal K}(t,\kappa)$, set  
    \begin{equation}\label{eq:CD1-Def}
      \mathcal X_{\kappa'}^{t,\kappa}
        :=\sum_{\kappa''\in {\rm AN}(\kappa,\kappa')\setminus\{\kappa'\}}
            X_{\kappa''} \bigl(T_{\kappa''}^{t,\kappa}\bigr),
      \qquad
      \mathcal X_{\kappa}^{t,\kappa}:=0 .
    \end{equation}
    Because $|X_{\kappa'} (s)|\le s$, the following classical ``light cone'' estimates follow directly:
    \begin{equation}\label{eq:CD1-Cone}
      \bigl|\mathcal X_{\kappa'}^{t,\kappa}\pm
             T_{\kappa'}^{t,\kappa}\bigr|\le t,\qquad
      \bigl|\mathcal X_{\kappa'}^{t,\kappa}
      +
      X_{\kappa'} \bigl(T_{\kappa'}^{t,\kappa}\bigr)\bigr|\le t,
      \qquad\kappa'\in{\cal K}(t,\kappa).
    \end{equation}

\item {$d=2,3$.} 
    For every $\kappa'\in{\cal K}(t,\kappa)$, define  
    \begin{equation}
    \label{eq:CDhigh-Def}
      \mathcal X_{\kappa'}^{t,\kappa}
        :=\sum_{\kappa''\in {\rm AN}(\kappa,\kappa')}
            X_{\kappa''} \bigl(T_{\kappa''}^{t,\kappa}\bigr),
      \qquad
      \mathcal X_{\kappa}^{t,\kappa}:=
      X_{\kappa} \bigl(T_{\kappa}^{t,\kappa}\bigr). 
    \end{equation}
    If $\kappa'$ is a parent and $(\kappa' , j)$ its $j$-th child, then
    \begin{equation}
      \nonumber %
      \mathcal X_{( \kappa' , j)}^{t,\kappa}
        =\mathcal X_{\kappa'}^{t,\kappa}
         +X_{( \kappa' , j)} \bigl(T_{(\kappa' , j)}^{t,\kappa}\bigr),
      \qquad j=1,\dots,J_{\kappa'}.
    \end{equation}
    Again, $|X_{\kappa'} (s)|\le s$ yields the cone bound  
    \begin{equation}\label{eq:CDhigh-Cone}
      \bigl|\mathcal X_{\kappa'}^{t,\kappa}\bigr|\le t,
      \qquad\kappa'\in{\cal K}(t,\kappa).
    \end{equation}
\end{enumerate}

\subsection{Sample realisation}
\noindent
Figure~\ref{sample-tree} shows a sample labelled random tree rooted at $(1)$ and
truncated at time $t$. Each node $\kappa$ survives for a time
$T^{t,(1)}_\kappa $ (for interior nodes $T^{t,(1)}_\kappa =\tau_\kappa $).
Every node carries a spatial mark $X_\kappa \bigl(T^{t,(1)}_\kappa \bigr)$,
which contributes to the cumulative displacement $\mathcal X^{t,(1)}_\kappa $
when $d\in\{2,3\}$ (cf.\ \eqref{eq:CDhigh-Def}), or to the cumulative
displacement of its descendants when $d=1$ (cf.\ \eqref{eq:CD1-Def}). 
In this illustration we take
$f(u)=a_0+a_2 u^2+a_3 u^3$ and choose $q_0,q_2,q_3>0$,
 with $q_0+q_2+q_3 = 1$, 
 so that each interior
 node, upon death, produces $0$, $2$, or $3$ random offsprings
 using the distribution $(q_j)_{j\in\{0,2,3\}}$.

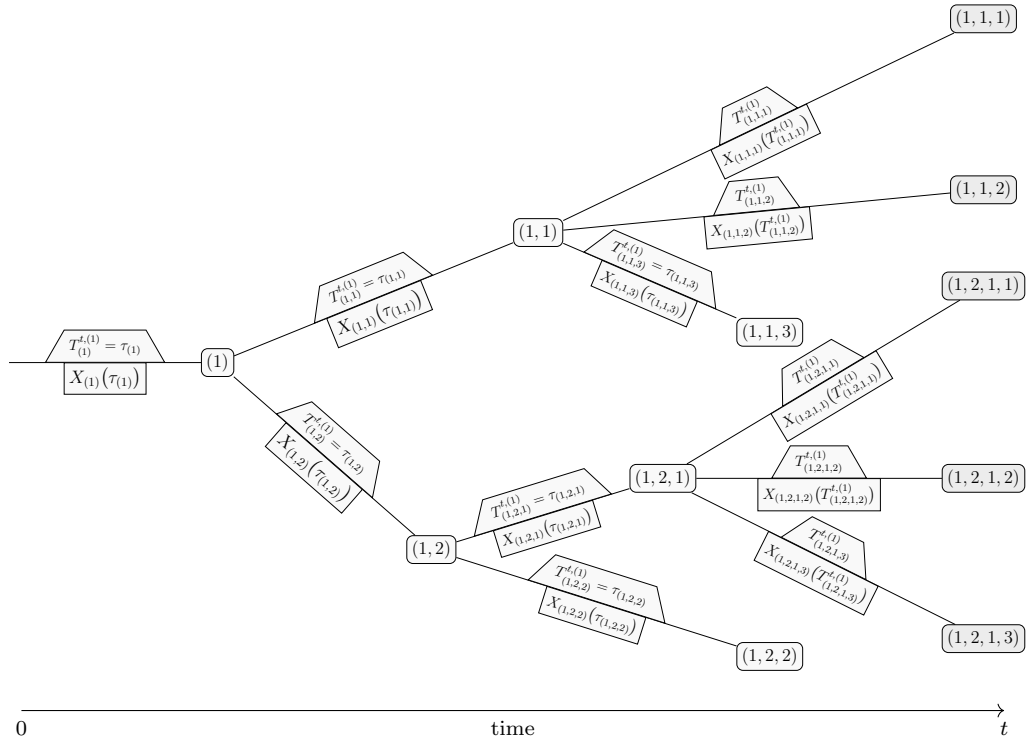
\begin{figure}[H]
\centering
\tikzstyle{level 1}=[level distance=6cm, sibling distance=6cm]
\tikzstyle{level 2}=[level distance=7cm, sibling distance=6.5cm]
\tikzstyle{level 3}=[level distance=8.5cm, sibling distance=5cm]
\tikzstyle{level 4}=[level distance=7.5cm, sibling distance=4cm]
\begin{center}
\resizebox{0.83\textwidth}{!}{
\begin{tikzpicture}[scale=1.0,grow=right, sloped]
    \node[]{} 
    child
    {
        node[name=data, fill=gray!05, draw, rounded corners=5pt] (main){\large $(1)$}
        child
        {
            node[draw,fill=gray!05,text=black,thick,xshift=-1cm,yshift=-2cm,rounded corners=5pt]{\large $(1,2)$} 
            child
            {
                node[name=data,draw, fill=gray!05,rounded corners=5pt,xshift=1cm,yshift=-0.5cm]{\large $(1,2,2)$}
                edge from parent
                node[above,trapezium,draw,fill=gray!05]{$T^{t,(1)}_{(1,2,2)}=\tau_{(1,2,2)}$}
                node[rectangle,draw,fill=gray!05,below]{$X_{(1,2,2)}\bigl(\tau_{(1,2,2)}\bigr)$}
            }
            child
            {
                node[draw,black,text=black,thick,xshift=-2cm,yshift=-0.5cm,fill=gray!05,rounded corners=5pt]{\large $(1,2,1)$} 
                child
                {
                    node[name=data,draw,fill=gray!15, rounded corners=5pt,xshift=1.5cm,yshift=-0.5cm]{\large $(1,2,1,3)$}
                    edge from parent
                    node[above,trapezium,draw,fill=gray!05]{$T^{t,(1)}_{(1,2,1,3)}$}
                    node[rectangle,draw,fill=gray!05,below]{$X_{(1,2,1,3)}\bigl(T^{t,(1)}_{(1,2,1,3)}\bigr)$}
                }
                child
                {
                    node[name=data,draw,fill=gray!15, rounded corners=5pt,xshift=1.5cm,yshift=0cm]{\large $(1,2,1,2)$}
                    edge from parent
                    node[above,trapezium,draw,fill=gray!05]{$T^{t,(1)}_{(1,2,1,2)}$}
                    node[rectangle,draw,fill=gray!05,below]{$X_{(1,2,1,2)}\bigl(T^{t,(1)}_{(1,2,1,2)}\bigr)$}
                }
                child
                {
                    node[name=data,draw,fill=gray!15, rounded corners=5pt,xshift=1.5cm,yshift=1.4cm]{\large $(1,2,1,1)$}
                    edge from parent
                    node[above,trapezium,draw,fill=gray!05]{$T^{t,(1)}_{(1,2,1,1)}$}
                    node[rectangle,draw,fill=gray!05,below]{$X_{(1,2,1,1)}\bigl(T^{t,(1)}_{(1,2,1,1)}\bigr)$}
                }
                edge from parent
                node[above,trapezium,draw,fill=gray!05]{$T^{t,(1)}_{(1,2,1)}=\tau_{(1,2,1)}$}
                node[rectangle,draw,fill=gray!05,below]{$X_{(1,2,1)}\bigl(\tau_{(1,2,1)}\bigr)$}
            }
            edge from parent
            node[above,trapezium,draw,fill=gray!05]{$T^{t,(1)}_{(1,2)}=\tau_{(1,2)}$}
            node[rectangle,draw,fill=gray!05,below]{\large $X_{(1,2)}\bigl(\tau_{(1,2)}\bigr)$}
        }
        child
        {
            node[draw,black,text=black,thick,yshift=0.4cm,xshift=2cm,fill=gray!05,rounded corners=5pt]{\large $(1,1)$}
            child
            {
                node[draw,black,text=black,thick,xshift=-2cm,yshift=2.2cm,fill=gray!05,rounded corners=5pt]{\large $(1,1,3)$}
                edge from parent
                node[above,trapezium,draw,fill=gray!05]{$T^{t,(1)}_{(1,1,3)}=\tau_{(1,1,3)}$}
                node[rectangle,draw,fill=gray!05,below]{$X_{(1,1,3)}\bigl(\tau_{(1,1,3)}\bigr)$}
            }
            child
            {
                node[draw,black,text=black,thick,xshift=4cm,yshift=1.2cm,fill=gray!15,rounded corners=5pt]{\large $(1,1,2)$}
                edge from parent
                node[above,trapezium,draw,fill=gray!05
                ]{$T^{t,(1)}_{(1,1,2)}$}
                node[rectangle,draw,fill=gray!05,below]{$X_{(1,1,2)}\bigl(T^{t,(1)}_{(1,1,2)}\bigr)$}
            }
            child
            {
                node[name=data, draw, fill=gray!15, yshift=1cm,xshift=4cm,rounded corners=5pt] {\large $(1,1,1)$}
                edge from parent
                node[above,trapezium,draw,fill=gray!05]{$T^{t,(1)}_{(1,1,1)}$}
                node[rectangle,draw,fill=gray!05,below]{$X_{(1,1,1)}\bigl(T^{t,(1)}_{(1,1,1)}\bigr)$}
            }
            edge from parent
            node[above,trapezium,draw,fill=gray!05]{$T^{t,(1)}_{(1,1)}=\tau_{(1,1)}$}
            node[rectangle,draw,fill=gray!05,below]{\large $X_{(1,1)}\bigl(\tau_{(1,1)}\bigr)$}
        }
        edge from parent
        node[above,trapezium,draw,fill=gray!05] {$T^{t,(1)}_{(1)}=\tau_{(1)}$}
        node[below,rectangle,draw,fill=gray!05] {\large $X_{(1)}\bigl(\tau_{(1)}\bigr)$}
    };
\end{tikzpicture}
}
\begin{tikzpicture}[%
    every node/.style={
        font=\scriptsize,
        text height=1ex,
        text depth=.25ex,
    },
]
\draw[->] (-4.5,-0.5) -- (8.5,-0.5);
\node[anchor=north] at (0,0) { };
\node[anchor=north] at (-4.5,-0.5) {$0$};
\node[anchor=north] at (2,-0.5) {time};
\node[anchor=north] at (8.5,-0.5) {$t$};
\end{tikzpicture}
\end{center}
\vskip-0.3cm
\caption{Sample labelled random tree rooted at $(1)$, truncated at time $t$.}
\label{sample-tree}
\end{figure}

\subsection{Random functional}
\noindent {\bf Node weights.}
Let $\kappa \in \mathbb{K}$.
For any $z\in \C$, $t\geq 0$
and boundary node \(\kappa'\in {\cal K}^{\rm b} (t,\kappa)\), we
define the boundary weight 
\begin{align}
  \label{eq:Ipartial}
  & W_{\kappa}^{\rm b} (z,t,\kappa')
  \\
  \nonumber
  & \quad
  :=
  \begin{cases}
    \displaystyle
    \frac{1}{2} 
    \bigl(
    \phi(z+c(\mathcal X_{\kappa'} +T_{\kappa'} ))
    +\phi(z+c(\mathcal X_{\kappa'} -T_{\kappa'} ))\bigr)
    +T_{\kappa'}  
    \psi(z+c(\mathcal X_{\kappa'} +X_{\kappa'} (T_{\kappa'} ))), & d=1,
    \\[8pt]
\nonumber
    \displaystyle
         \phi(z+c\mathcal X_{\kappa'} )
     +\nabla \phi(z+c\mathcal X_{\kappa'} ) \cdot cX_{\kappa'} (T_{\kappa'} )
     +T_{\kappa'}  
       \psi(z+c\mathcal X_{\kappa'} ), & d=2,3,
\end{cases}
\end{align} 
where for brevity we let 
\(
  \mathcal X_{\kappa'} =\mathcal X_{\kappa'}^{t,\kappa}, 
  T_{\kappa'} =T_{\kappa'}^{t,\kappa}.
\)
For any interior node \(\kappa'\in{\cal K}^\circ(t,\kappa)\), we also put
\begin{equation}
  \label{eq:Io}
  W_{\kappa}^\circ (t,\kappa')
    :=\frac{\tau_{\kappa'} }{\lambda} 
      \frac{a_{J_{\kappa'} }}{q_{J_{\kappa'} }}.
\end{equation}

\begin{definition}
\noindent %
 For $z\in \C^d$, $t\geq 0$ and $\kappa \in \mathbb{K}$, 
 we let $ \mathcal{H}(z,t,\kappa)$ be the random functional
 defined as 
 \begin{equation}
   \label{eq:GlobalWeight}
  {%
  \mathcal{H}(z,t,\kappa)
     :=\prod_{\kappa'\in{\cal K}(t,\kappa)}
             e^{\lambda T_{\kappa'}^{t,\kappa}} 
        \prod_{\kappa'\in{\cal K}^{\rm b} (t,\kappa)}
             W_{\kappa}^{\rm b} (z,t,\kappa') 
        \prod_{\kappa'\in{\cal K}^\circ(t,\kappa)}
             W_{\kappa}^\circ ( t, \kappa')}
\end{equation}
 generated by the subtree rooted at $\kappa$ and truncated at time $t$.
\end{definition}
\noindent 
 The three products in \eqref{eq:GlobalWeight}
 correspond to {exponential survival time of each node}, {boundary nodes contributions} and {interior nodes branching weights}, respectively.
 We also use the abbreviation 
\(
  \mathcal{H}(z,t):=\mathcal{H}(z,t,(1))
\).

\subsection{Branching recursion}

\noindent Following from \eqref{eq:GlobalWeight}, we
derive a branching recursion based on first-step decomposition at a generic node $\kappa$.
\begin{lemma} 
  Let $\kappa \in \mathbb{K}$ and $t\geq 0$.
  We have the relations 
  \begin{enumerate}[i)] %
  \item For $d=1$, we have 
    \begin{align}
\nonumber   
    \mathcal{H}(z,t,\kappa) & = \mathbf1_{\{\tau_\kappa  \ge t\}} e^{\lambda t} \biggl(
    \frac{1}{2}\phi(z+ct) + \frac{1}{2}\phi(z-ct) + t\psi(z+cX_\kappa (t))\biggl)
    \\
    \label{eq:Recursion1D}
    & \quad + \mathbf1_{\{\tau_\kappa  < t\}} e^{\lambda \tau_\kappa } \frac{\tau_\kappa }{\lambda} \frac{a_{J_\kappa}}{q_{J_\kappa}} \mathop{{\prod_{j =1}^{J_\kappa}}} \mathcal{H}\big(z+cX_\kappa (\tau_\kappa),t-\tau_\kappa , ( \kappa , j )\big). %
\end{align}
\item For $d=2,3$, we have
  \begin{align}
    \nonumber
    \mathcal{H}(z,t,\kappa) & = \mathbf1_{\{\tau_\kappa  \geq t\}} e^{\lambda t} \bigl(
    \phi(z+cX_\kappa (t)) + \nabla \phi(z+cX_\kappa (t))\cdot(cX_\kappa (t)) + t\psi(z+cX_\kappa (t))\bigl)
    \\
    \label{eq:RecursionHighD}
    & \quad + \mathbf1_{\{\tau_\kappa  < t\}} e^{\lambda \tau_\kappa } \frac{\tau_\kappa }{\lambda} \frac{a_{J_\kappa}}{q_{J_\kappa}} \mathop{{\prod_{j =1}^{J_\kappa}}} \mathcal{H}\big(z+cX_\kappa (\tau_\kappa),t-\tau_\kappa , ( \kappa , j ) \big). %
\end{align}
\end{enumerate}
\end{lemma} 
\begin{proof}
We distinguish the two cases according to the value of the first lifetime $\tau_\kappa$.

\medskip
\noindent\textbf{Case 1: $\tau_\kappa\ge t$.}
In this case, the root particle $\kappa$ survives beyond the truncation time $t$, hence it belongs to ${\cal K}^{\rm b} (t,\kappa)$, and no branching occurs before time $t$, hence $T_\kappa^{t,\kappa}=t$.
In addition, the cumulative displacement at the root
is $\mathcal X_\kappa^{t,\kappa}=0$ when $d=1$ (see
\eqref{eq:CD1-Def}), while for $d=2,3$ we have
$\mathcal X_\kappa^{t,\kappa}=X_\kappa(T_\kappa^{t,\kappa})=X_\kappa(t)$ (see
\eqref{eq:CDhigh-Def}). Recalling that \eqref{eq:GlobalWeight}
only contains $e^{\lambda t}$ and a
single boundary weight given by \eqref{eq:Ipartial},
we obtain \eqref{eq:Recursion1D} for $d=1$ and
 \eqref{eq:RecursionHighD} for $d=2,3$.

\medskip
\noindent\textbf{Case 2: $\tau_\kappa<t$.}
In this case the particle $\kappa$ dies before time $t$, hence
it belongs to ${\cal K}^\circ (t,\kappa)$, 
  and we have $T_\kappa^{t,\kappa}=\tau_\kappa$. Its contribution
to \eqref{eq:GlobalWeight} is therefore the factor $e^{\lambda\tau_\kappa}$
together with the interior weight 
$$W^\circ_\kappa(t,\kappa)=\frac{\tau_\kappa}{\lambda}\frac{a_{J_\kappa}}{q_{J_\kappa}}
$$ 
 in \eqref{eq:Io}. 
At the death time $\tau_\kappa$, the process branches into $J_\kappa$ i.i.d.\
subtrees rooted at $(\kappa,j)$ for $j=1,\dots,J_\kappa$,
each started from the shifted space-time point
$(z+cX_\kappa(\tau_\kappa),\,t-\tau_\kappa)$. Since \eqref{eq:GlobalWeight} is a
product over node contributions,
 $\mathcal{H}(z,t,\kappa)$ factors as the product of
the root contribution
$$
e^{\lambda\tau_\kappa}
W^\circ_\kappa(t,\kappa)=
e^{\lambda\tau_\kappa}
\frac{\tau_\kappa}{\lambda}\frac{a_{J_\kappa}}{q_{J_\kappa}}
$$
 and the product over $j=1,\dots,J_\kappa$ of
 sub-branches functionals,
 $\mathcal{H}\big(z+cX_\kappa (\tau_\kappa),t-\tau_\kappa , ( \kappa , j ) \big)$,
 which yields the second terms in \eqref{eq:Recursion1D} and
\eqref{eq:RecursionHighD}.
\end{proof}

\subsubsection*{Algorithms} 

\noindent
The branching recursions \eqref{eq:Recursion1D}-\eqref{eq:RecursionHighD} are implemented in the following algorithms.

\begin{algorithm}[H]
  \footnotesize
  \caption{\footnotesize Monte Carlo estimation in one spatial dimension ($d=1$).}
\label{alg:d1}
\begin{algorithmic}%
\Require time $t$, spatial point $z$, sample size $n$, rate $\lambda$, coefficients $(a_k)_{k=n}^N$
\State Build probability array $\textbf{q}[0{:}N]$ with
       $q_j>0 \Leftrightarrow a_j\neq0$ and $\sum_{k=0}^N q_k=1$
\State \textbf{for} $i=1,\dots,n$ \textbf{do}   \Comment{outer MC loop}
        \State \hspace{\algorithmicindent}%
               $\text{arr}[i] \gets \textsc{Branch1D}(z,t)$
\State \textbf{return} $\displaystyle\widehat v=\frac1n\sum_{i=1}^n\text{arr}[i]$
\medskip
\Function{Branch1D}{$z, t$}
        \State Draw $\tau\sim\mathrm{Exp}(\lambda)$, $p\sim\mathrm{Unif}[0,1]$
        \If{$\tau\ge t$}                          \Comment{no offspring}
              \State \Return $e^{\lambda t} \Bigl(
                      \frac{1}{2}\phi(z+ct)+\frac{1}{2}\phi(z-ct)+t\psi(z+ct(2p-1))
                      \Bigr)$
        \Else                                      \Comment{branching event}
              \State Draw $J\in\{j:q_j>0\}$ with $\PP (J=j)=q_j$
              \State $H\gets1$
              \For{$\ell=1$ \textbf{to} $J$}
                     \State $H\gets H\times\textsc{Branch1D}\bigl(z+c\tau(2p-1), t-\tau\bigr)$
              \EndFor
              \State \Return $e^{\lambda\tau} 
                     \dfrac{\tau}{\lambda} 
                     \dfrac{a_{J}}{q_{J}} H$
        \EndIf
\EndFunction
\end{algorithmic}
\end{algorithm}

\begin{algorithm}[H]
  \footnotesize
  \caption{\footnotesize Monte Carlo estimation in two spatial dimensions ($d=2$).}
\begin{algorithmic}[1]
\Require $t$, $(z_1,z_2)$, $n$, $\lambda$, $(a_j)$
\State Initialise $\textbf{q}[0{:}N]$ as in Alg.~\ref{alg:d1}
\State \textbf{for} $i=1,\dots,n$ \textbf{do} \quad
       $\text{arr}[i]\gets\textsc{Branch2D}(z_1,z_2,t)$
\State \textbf{return} $\widehat v=\frac1n\sum_i\text{arr}[i]$
\medskip
\Function{Branch2D}{$z_1,z_2,t$}
        \State Draw $\tau\sim\mathrm{Exp}(\lambda)$,
               $p\sim\mathrm{Unif}[0,1]$, $\theta\sim\mathrm{Unif}[0,2\pi]$
        \If{$\tau\ge t$}
              \State $R\gets t\sqrt{1-(1-p)^2}$,
                     $y_1\gets cR\cos\theta$, $y_2\gets cR\sin\theta$
              \State $I_1\gets\phi(z_1+y_1,z_2+y_2)$
              \State $I_2\gets y_1 \partial_{z_1}\phi(z_1+y_1,z_2+y_2)+ y_2 \partial_{z_2}\phi(z_1+y_1,z_2+y_2)$
              \State $I_3\gets t \psi(z_1+y_1,z_2+y_2)$
              \State \Return $e^{\lambda t}(I_1+I_2+I_3)$
        \Else
              \State $R\gets \tau\sqrt{1-(1-p)^2}$,
                     $y_1\gets cR\cos\theta$, $y_2\gets cR\sin\theta$
              \State Draw $J$ with $\PP (J=j)=q_j$;  $H\gets1$
              \For{$\ell=1$ \textbf{to} $J$}
                     \State $H\gets H\times
                       \textsc{Branch2D}(z_1+y_1,z_2+y_2,t-\tau)$
              \EndFor
              \State \Return $e^{\lambda\tau}
                     \dfrac{\tau}{\lambda}
                     \dfrac{a_J}{q_J} H$
        \EndIf
\EndFunction
\end{algorithmic}
\end{algorithm}

\vspace{-0.3cm}

\noindent
 In the following algorithm we consider $\arccos$ as
 a function from $[-1,1]$ to $[0,\pi]$.
\begin{algorithm}[H]
  \footnotesize
  \caption{\footnotesize Monte Carlo estimation in three spatial dimensions ($d=3$).}
\begin{algorithmic}[1]
\Require $t$, $(z_1,z_2,z_3)$, $n$, $\lambda$, $(a_j)$
\State Prepare $\textbf{q}$ as before
\State \textbf{for} $i=1,\dots,n$ \textbf{do}\quad
       $\text{arr}[i]\gets\textsc{Branch3D}(z_1,z_2,z_3,t)$
\State \textbf{return} $\widehat v=\frac1n\sum_i\text{arr}[i]$
\medskip
\Function{Branch3D}{$z_1,z_2,z_3,t$}
        \State Draw $\tau\sim\mathrm{Exp}(\lambda)$,
               $p\sim\mathrm{Unif}[0,1]$, $\theta\sim\mathrm{Unif}[0,2\pi]$
        \If{$\tau\ge t$}
              \State $\alpha\gets\arccos(1-2p)$ %
              \State $y_1\gets ct\sin\alpha\cos\theta$,
                     $y_2\gets ct\sin\alpha\sin\theta$,
                     $y_3\gets ct\cos\alpha$
              \State $I_1\gets\phi(z_1+y_1,z_2+y_2,z_3+y_3)$
              \State $I_2 \gets
                      \begin{aligned}[t]
                          &y_1 \partial_{z_1}\phi(z_1+y_1, z_2+y_2, z_3+y_3)+y_2 \partial_{z_2}\phi(z_1+y_1, z_2+y_2, z_3+y_3)\\
\nonumber
                                              &+y_3 \partial_{z_3}\phi(z_1+y_1, z_2+y_2, z_3+y_3)
                      \end{aligned}$
              \State $I_3\gets t \psi(z_1+y_1,z_2+y_2,z_3+y_3)$
              \State \Return $e^{\lambda t}(I_1+I_2+I_3)$
        \Else
              \State $\alpha\gets\arccos(1-2p)$ %
              \State $y_1\gets c\tau\sin\alpha\cos\theta$,
                     $y_2\gets c\tau\sin\alpha\sin\theta$,
                     $y_3\gets c\tau\cos\alpha$
              \State Draw $J$ with $\PP (J=j)=q_j$;  $H\gets1$
              \For{$\ell=1$ \textbf{to} $J$}
                     \State $H\gets H\times
                       \textsc{Branch3D}(z_1+y_1,z_2+y_2,z_3+y_3,t-\tau)$
              \EndFor
              \State \Return $e^{\lambda\tau}
                     \dfrac{\tau}{\lambda}
                     \dfrac{a_J}{q_J} H$
        \EndIf
\EndFunction
\end{algorithmic}
\end{algorithm}

\section{Mild solutions}
\label{s4}
\begin{definition}[Admissible space-time cones]
  Let
  \(t_0>0\).
   \begin{enumerate}[i)] %
    \item
    Forward light cone %
    in $\R^d$.
    Given \(x_0\in\mathbb{R}^d\), let %
$$ 
 \Gamma_{x_0,t_0}(\R^d\times[0,t_0])
  :=\bigl\{(y,s)\in\R^d\times[0,t_0]
 :|y-x_0|\le t_0-s\bigr\}.
 $$
\item
  Forward light cone %
  in $\C^d$.
  Given \(z_0\in\mathbb{C}^d\), let 
$$ 
 \widebar{\Gamma}_{z_0,t_0}(\C^d\times[0,t_0])
   :=\bigl\{(z,s)\in\C^d\times[0,t_0]
:\exists y\in\R^d \text{ s.t. }(y,s)\in
\Gamma_{0,t_0}(\R^d\times[0,t_0])
            \text{ and }z=z_0+cy\bigr\}.
                          $$
             \end{enumerate} 
\end{definition} 
\noindent
 The space-time cone $\widebar{\Gamma}_{z_0,t_0}(\C^d)$
 defined in \eqref{space-time-cone} is the spatial footprint of
 $\widebar{\Gamma}_{z_0,t_0}(\C^d\times[0,t_0])$ at time~$t_0$. 
\noindent
 We will work under the following conditions.
   \begin{assumption}
   \label{assu1}
   Let $(z_0,t_0)\in\C^d\times(0,\infty)$,
   and assume that $\phi, \psi$
   are bounded measurable functions on
   $\widebar{\Gamma}_{z_0,t_0}(\C^d)$, with in addition 
  $\phi\in {\cal C}^1 \big(
  \widebar{\Gamma}_{z_0,t_0}(\C^d)
  \big)$ if $d=2,3$. 
\end{assumption}
\noindent
  The condition $\phi\in {\cal C}^1 \big(
  \widebar{\Gamma}_{z_0,t_0}(\C^d)
  \big)$
  in Assumption~\ref{assu1} 
  will be required for the application of 
 \eqref{qakjfds-1}-\eqref{qakjfds-2}
 when $d=2,3$.
\begin{theorem}[Existence of mild solutions]\label{thm:Mild}
  Let $(z_0,t_0)\in\C^d\times(0,\infty)$.
  Suppose that Assumption~\ref{assu1} holds, and that
 \begin{equation}
   \label{it:Int}
   \E \bigl[|\mathcal{H}(z,t)|\bigr]<\infty,
   \quad (z,t)\in\widebar{\Gamma}_{z_0,t_0}(\C^d\times[0,t_0]). 
 \end{equation}
 Then, {the function} 
\begin{equation}
\nonumber %
v(z,t):=\E \bigl[\mathcal{H}(z,t)\bigr], \quad 
 (z,t)\in\widebar{\Gamma}_{z_0,t_0}(\C^d\times[0,t_0]),
\end{equation}
 yields a {mild solution}
 of \eqref{eq:WaveMain}-\eqref{eq:PolyNonlin},
 which satisfies the integral identity \eqref{eq:Mild1D} for $d=1$,
 resp. \eqref{eq:Mild2Dpolar}, \eqref{eq:Mild3Dpolar} for $d=2,3$. 
\end{theorem}
\begin{proof}
The proof is presented for $d=1$,
the cases $d=2,3$ being analogous
by replacing \eqref{eq:Mild1D}
 with \eqref{qakjfds-1}-\eqref{qakjfds-2}. 
Condition~\eqref{it:Int} implies that the expectation $v(z,t)=\E[\mathcal{H}(z,t)]$ exists for every
$(z,t)\in\widebar{\Gamma}_{z_0,t_0}(\C^d\times[0,t_0])$,
 hence $v$ is well defined. Set
\(
  \mathcal F_{(1)}
  :=\sigma \bigl(\tau_{(1)},J_{(1)},U_{(1)}\bigr),
\)
the information carried by the root node. Using the tower property and the branching recursion \eqref{eq:Recursion1D} we obtain
\begin{align}
\nonumber %
v(z,t)
   &=\E \bigl[\mathcal{H}(z,t,(1))\bigr]
\\
\nonumber
&=\E\Biggl[
      \mathbf1_{\{\tau_{(1)}\ge t\}}e^{\lambda t}
         \Bigl(\frac{1}{2}\phi(z + ct)+\frac{1}{2}\phi(z - ct)
         +t \psi\bigl(z+cX_{(1)}(t)\bigr)\Bigr)
         \\
\nonumber
   &\quad+\mathbf1_{\{\tau_{(1)}< t\}}e^{\lambda\tau_{(1)}}
        \frac{\tau_{(1)}}{\lambda}
           \frac{a_{J_{(1)}}}{q_{J_{(1)}}}
           \prod_{j =1}^{J_{(1)}}
              \mathcal{H}\bigl(
                 z+cX_{(1)}(\tau_{(1)}), 
                 t-\tau_{(1)}, 
                 (1 , j) \bigr)
                 \Biggr]
\\
\label{eq:FirstStepCond} 
   &=\E\Biggl[
      \mathbf1_{\{\tau_{(1)}\ge t\}}e^{\lambda t}
         \Bigl(\frac{1}{2}\phi(z + ct)+\frac{1}{2}\phi(z - ct)
         +t \psi\bigl(z+cX_{(1)}(t)\bigr)\Bigr)
         \\
\nonumber
   &\quad+\mathbf1_{\{\tau_{(1)}< t\}}e^{\lambda\tau_{(1)}}
        \frac{\tau_{(1)}}{\lambda}
           \frac{a_{J_{(1)}}}{q_{J_{(1)}}}
           \E\Biggl[
             \prod_{j =1}^{J_{(1)}}
              \mathcal{H}\bigl(
                 z+cX_{(1)}(\tau_{(1)}), 
                 t-\tau_{(1)}, 
                 (1 , j)\bigr)
                 \bigg| \mathcal{F}_{(1)}
                 \Biggr]
           \Biggr], 
\end{align}
\noindent
 where we used the facts that $\tau_{(1)},J_{(1)},U_{(1)}$ are
 $\mathcal F_{(1)}$-measurable.
When $\tau_{(1)}<t$ and $J_{(1)}\ge1$, independence and identical distribution of the $J_{(1)}$ sub-branches give
\begin{align*} 
  \nonumber %
  \E \Biggl[ \prod_{j =1}^{J_{(1)}}
       \mathcal{H}\bigl(
          z+cX_{(1)}(\tau_{(1)}), 
          t-\tau_{(1)}, 
          (1 , j) \bigr)
     \biggm|\mathcal F_{(1)}\Biggr]
  & =
  \prod_{j =1}^{J_{(1)}}
  \E \Biggl[ 
       \mathcal{H}\bigl(
          z+cX_{(1)}(\tau_{(1)}), 
          t-\tau_{(1)}, 
          (1, j) \bigr)
     \biggm|\mathcal F_{(1)}\Biggr]
  \\
\nonumber
& =\bigl(v(z+cX_{(1)}(\tau_{(1)}), t-\tau_{(1)})\bigr)^{J_{(1)}},
\end{align*} 
so that the integrand in \eqref{eq:FirstStepCond} depends only on $\tau_{(1)},J_{(1)},U_{(1)}$.

\medskip\noindent
\textbf{Continuing from~\eqref{eq:FirstStepCond}.}
Taking expectations in \eqref{eq:FirstStepCond} in the following order-(i) first with respect to the
offspring count by summing over all possibilities of \(J_{(1)}\), i.e.
\(\sum_{j:a_j\neq0} q_j\,(\cdot)\) with \(q_j=\PP (J_{(1)}=j)\); (ii) then conditioning on $\tau_{(1)}=s>0$ and averaging over $X_{(1)}(s)$, which is uniformly distributed on $[-s,s]$ with density $(2s)^{-1}$; and (iii) finally integrating in $s>0$ against the exponential density $\lambda e^{-\lambda s}$ of $\tau_{(1)}$
 yields
\begin{equation}
\nonumber %
v(z,t)
= \frac{1}{2}(\phi(z+ct)+\phi(z-ct)) 
   + \frac{1}{2} \int_{-t}^{t}\!\psi(z+cy)\,dy  
  + \frac{1}{2} \int_{0}^{t}\!\int_{-s}^{s}
     \sum_{k=0}^N %
     a_k v(z+cy,t-s)^k \,dy\,ds,
\end{equation}
which coincides with the one-dimensional mild formulation
 \eqref{eq:Mild1D} with $u$ replaced by~$v$.
\end{proof}
\noindent
 As a consequence of Theorem~\ref{thm:Mild} 
 and Propositions~\ref{prop:N1Lp} and \ref{thm:Nge2Lp} below
 applied for $p=1$, we have the following result.
\begin{theorem}[Existence of mild solutions]\label{thm:mild-existence}
 Let $N \geq 1$ and $(z_0,t_0)\in\C^d \times (0,\infty)$. 
Suppose that Assumption~\ref{assu1} holds,
and that
\begin{equation}
\label{jfklfdsa1a} 
 e^{\lambda t_0N} -e^{\lambda t_0}
 < 
 \frac{1}{\max (1 , C_{\rm b}(t_0,z_0) )^{ N-1}
 \max (1 , t_0 C_\circ /\lambda )} \ \mbox{~if~} N\geq 2, 
\end{equation}
 where we set 
\begin{equation}\label{eq:Apartial0}
C_\circ := \max_{\substack{0\le j\le N\\ a_j\neq 0}} \frac{|a_j|}{q_j}
\end{equation} 
 and %
\begin{equation}\label{eq:Apartial}
C_{\rm b}(t_0,z_0) :=
\begin{cases}
\displaystyle
\sup_{y\in
  \widebar{\Gamma}_{z_0,t_0}(\C^d)
  } |\phi(y)|
\;+\;
t_0\,\sup_{y\in
  \widebar{\Gamma}_{z_0,t_0}(\C^d)
  } |\psi(y)|,
& d=1,\\[3ex]
\displaystyle
\sup_{y\in \widebar{\Gamma}_{z_0,t_0}(\C^d)
  }
\bigl(|\phi(y)| + t_0\,|\psi(y)| + |c|\,t_0\,\|\nabla \phi(y)\|\bigr),
& d=2,3.
\end{cases}
\end{equation}
 Then, for every $(z,t)\in\widebar{\Gamma}_{z_0,t_0}(\C^d\times[0,t_0])$
 and $p\geq 1$
 we have 
$
 \E [ |\mathcal{H}(z,t)| ] < \infty$,
 and the function %
 \begin{equation}
   \label{eq:Estimator}
   v(z,t):=\E \bigl[\mathcal{H}(z,t)\bigr],
   \quad
   (z,t)\in\widebar{\Gamma}_{z_0,t_0}(\C^d\times[0,t_0]), 
\end{equation}
 yields a mild solution
 of \eqref{eq:WaveMain}-\eqref{eq:PolyNonlin},
 which satisfies the integral identity \eqref{eq:Mild1D} for $d=1$,
 resp. \eqref{eq:Mild2Dpolar}, \eqref{eq:Mild3Dpolar} for $d=2,3$. 
\end{theorem}
\noindent
We note that, taking
 $\lambda := t_0$ and 
\[
q_k:=\frac{|a_k|}{|a_n|+ \cdots + |a_N|}, \qquad k=n,\ldots , N, 
\]
 Condition~\eqref{jfklfdsa1a} 
 in Theorem~\ref{thm:mild-existence} 
 can be replaced for $N\ge 2$ with 
\begin{equation}
\label{eq:explicit-threshold-1d-0}
e^{t_0^2 N} - e^{t_0^2}
< 
\frac{1}{
\max (1 , \|\phi\|_{L^\infty ( \C )}+t_0\|\psi\|_{L^\infty ( \C )} )^{\,N-1}
\;\sum_{k=n}^N |a_k|}
\end{equation}
 for $d=1$, and by  
\begin{equation}
\label{eq:explicit-threshold-23d-0}
e^{t_0^2 N} - e^{t_0^2}
< 
\frac{1}{
  \max (1 , \|\phi\|_{L^\infty ( \C^d )} +t_0
   \|\psi\|_{L^\infty ( \C^d )} 
+|c|\,t_0\,\|\nabla \phi\|_{L^\infty ( \C^d )} )^{\,N-1}
\;\sum_{k=n}^N |a_k|}
\end{equation}
 for $d=2,3$. 

\section{Integrability}
\label{s5} 
\noindent
Fix $p\in[1,\infty)$.
  Our aim is to impose {deterministic} restrictions on
  \begin{itemize}
  \item
    the initial data $\phi$, $\psi$,
  \item
    the degree $N$ of the polynomial nonlinearity, and
  \item the coefficients $(a_k)_{k=0}^N$
  \end{itemize}
  that ensure the $p$-$th$ integrability 
   \begin{equation}
     \nonumber %
  \E\bigl[ |\mathcal{H}(z,t)|^p \bigr] < \infty,
  \qquad  \forall (z,t)\in\widebar{\Gamma}_{z_0,t_0}(\C^d\times[0,t_0]), 
\end{equation}
 see Propositions~\ref{prop:N1Lp} and \ref{thm:Nge2Lp}. 
\noindent
 Since $q_j=0$
 if and only if $a_j=0$, the maximization in $C_\circ=\max_{0\le j\le N,\;a_j\ne0}|a_j|/q_j$ ranges only over indices with $q_j>0$, so $C_\circ<\infty$.
 Moreover, under Assumption~\ref{assu1} the functions
 $\phi$, $\psi$, and $\nabla \phi$ for $d=2,3$,
 are bounded measurable on the compact set
$\widebar{\Gamma}_{z_0,t_0}(\C^d)
$; thus $C_{\rm b}(t_0,z_0) <\infty$. 

 \medskip

The cone bounds \eqref{eq:CD1-Cone} and \eqref{eq:CDhigh-Cone} imply that every argument of $\phi,\psi$ (and $\nabla \phi$ for $d=2,3$) appearing in \eqref{eq:Ipartial} lies in $\widebar{\Gamma}_{z_0,t_0}(\C^d)
$. Hence, by \eqref{eq:Ipartial}-\eqref{eq:Io} and
 \eqref{eq:Apartial0}-\eqref{eq:Apartial} we have 
\[
|W^{\rm b}_{(1)}(z,t,\kappa)| \le C_{\rm b}(t_0,z_0)
\quad
\mbox{and}
\quad 
  |W^\circ_{(1)}(t,\kappa)| \le \frac{t_0}{\lambda}C_\circ, 
\]
 $(z,t)\in
\widebar{\Gamma}_{z_0,t_0}(\C^d\times[0,t_0])$,
 $\kappa\in{\cal K}(t,(1))$. 
Inserting these estimates into \eqref{eq:GlobalWeight} yields the bound 
\begin{equation}\label{eq:Hbound}
  |\mathcal{H}(z,t)|
   \le 
  (C_{\rm b}(t_0,z_0))^{|{\cal K}_t^{\rm b} |} 
   \left( \frac{t_0}{\lambda}C_\circ \right)^{|{\cal K}_t^\circ |}
      \prod_{\kappa\in{\cal K}_t }e^{\lambda T_\kappa }
      ,
\end{equation}
 $     (z,t)\in
      \widebar{\Gamma}_{z_0,t_0}(\C^d\times[0,t_0])$,
 $t>0$, 
where we abbreviate $T_\kappa :=T_\kappa^{ t,(1)}$, ${\cal K}_t :={\cal K}(t,(1))$, and ${\cal K}_t^{\rm b} ,{\cal K}_t^\circ $ denote its boundary and interior subsets.

\subsubsection*{A dominating Galton--Watson tree}
\noindent
We introduce an auxiliary tree in which each particle {always} produces $N$ children, preserving the same exponential lifetime,
 with interior, boundary and complete node sets
 $\widetilde{{\cal K}}_{N,t}^\circ $,
 $\widetilde{{\cal K}}_{N,t}^{\rm b} $,
 and $\widetilde{{\cal K}}_{N,t}$, such that 
\[
  {\cal K}_t^\circ \subseteq\widetilde{{\cal K}}_{N,t}^\circ,
  \quad
  {\cal K}_t^{\rm b} \subseteq\widetilde{{\cal K}}_{N,t}^{\rm b} ,
  \quad
      {\cal K}_t \subseteq\widetilde{{\cal K}}_{N,t}. 
\]
\begin{lemma}[Boundary-interior relation in the $N$-ary dominating tree]
 For all $t\ge 0$, we have 
\begin{equation}\label{eq:DomRelation}
  \bigl|\widetilde{{\cal K}}_{N,t}^{\rm b} \bigr|
  =
  1+(N-1)\bigl|\widetilde{{\cal K}}_{N,t}^\circ \bigr|.
\end{equation}
\end{lemma}

\begin{proof}
We argue by induction on the number of branching events that occur before the
truncation time $t$ in the tree $\widetilde{{\cal K}}_{N,t}$.

\smallskip
\noindent\textbf{Base case.}
If no branching occurs before time $t$, then the tree consists only of the root,
which is a boundary node. Hence
$|\widetilde{{\cal K}}_{N,t}^{\rm b}|=1$ and
$|\widetilde{{\cal K}}_{N,t}^\circ |=0$, and \eqref{eq:DomRelation} holds.

\smallskip
\noindent\textbf{Induction step.}
Assume that after some number of branching events the relation
$|\widetilde{{\cal K}}_{N,t}^{\rm b}|=1+(N-1)|\widetilde{{\cal K}}_{N,t}^\circ |$
is satisfied, and consider the next branching event. At this event, a single
boundary node becomes an interior node and produces exactly $N$ new children,
all of which are boundary nodes. Therefore, the boundary count becomes 
\[
|\widetilde{{\cal K}}_{N,t}^{\rm b}|_{\rm new}
=
|\widetilde{{\cal K}}_{N,t}^{\rm b}|_{\rm old} + N-1,
\]
while the interior count increases by one, i.e. $
|\widetilde{{\cal K}}_{N,t}^\circ |_{\rm new}
=
|\widetilde{{\cal K}}_{N,t}^\circ |_{\rm old}+1$.
 Using the induction hypothesis, we find 
\begin{align*} 
|\widetilde{{\cal K}}_{N,t}^{\rm b}|_{\rm new}
 & =
\bigl(1+(N-1)|\widetilde{{\cal K}}_{N,t}^\circ |_{\rm old}\bigr)+(N-1)
\\
 & =
1+(N-1)\bigl(|\widetilde{{\cal K}}_{N,t}^\circ |_{\rm old}+1\bigr)
\\
 & =
1+(N-1)|\widetilde{{\cal K}}_{N,t}^\circ |_{\rm new},
\end{align*}
 which completes the induction.
\end{proof}
\noindent
 Next are our main integrability results.
\begin{proposition}
\label{prop:N1Lp}
Let $N=1$,
 $p\geq 1$, and $(z_0,t_0)\in\C^d \times (0,\infty)$. 
 Then, under Assumption~\ref{assu1},
 for every $(z,t)\in\widebar{\Gamma}_{z_0,t_0}(\C^d\times[0,t_0])$
 we have 
$$
 \E\bigl[ |\mathcal{H}(z,t)|^p \bigr] < \infty.
 $$ 
 \end{proposition}
\begin{proof}
 When $N=1$, $\widetilde{{\cal K}}_{N,t}$ is a Poisson birth-process
 with $|\widetilde{{\cal K}}_{N,t}^\circ |\sim\text{Poi}(\lambda t)$, whence
\begin{equation}
\label{eq:PoissonPGF}
  \E\!\bigl[s^{|\widetilde{{\cal K}}_{N,t}^\circ |}\bigr]
     =e^{ (s-1)\lambda t},
     \qquad s\ge0.
\end{equation}
 Also, 
 since the lifetimes of all nodes in ${\cal K}_t $ sum up to $\le t$, we have
 \(\prod_{\kappa\in{\cal K}_t }e^{\lambda T_\kappa }\le e^{\lambda t}\).
 Thus, \eqref{eq:Hbound} yields 
\begin{align}
\nonumber    
\E \big[ 
  |\mathcal{H}(z,t)|^p
  \big]
& \leq
  \E \left[ \max (1 , C_{\rm b}(t_0,z_0) )^{p|{\cal K}_t^{\rm b} |} 
   \max \left(1 , \frac{t_0}{\lambda}C_\circ \right)^{p|{\cal K}_t^\circ |}
   \left(
   \prod_{\kappa\in{\cal K}_t }e^{\lambda T_\kappa }
   \right)^p
   \right]
  \\
  \nonumber
  & \leq
  \E \left[ \max ( 1 , C_{\rm b}(t_0,z_0) )^p 
   \max \left(1 , \frac{t_0}{\lambda}C_\circ \right)^{p|\widetilde{{\cal K}}_{N,t}^\circ |}
   e^{p\lambda t} %
    \right]
  \\
\nonumber
    & =
  \max ( 1 , C_{\rm b}(t_0,z_0) )^p  
  e^{p\lambda t} 
  \E \left[ 
  \max \left( 1 , \frac{t_0}{\lambda}C_\circ \right)^{p|\widetilde{{\cal K}}_{N,t}^\circ |}
    \right]
  \\
\nonumber
     & =
 \max ( 1 , C_{\rm b}(t_0,z_0) )^p  
  e^{p\lambda t} 
  \exp \left(
\lambda t \left( 
\max \left( 1 , \frac{t_0}{\lambda} C_\circ \right)^p
-1 \right)\right)
    \\
\nonumber %
     & < \infty. 
\end{align} 
\end{proof} 
\begin{proposition}
\label{thm:Nge2Lp}
 Let $N \geq 2$, $p\geq 1$
 and $(z_0,t_0)\in\C^d \times (0,\infty)$. 
 Under Assumption~\ref{assu1}, if 
    \begin{equation}
      \label{eq:CriticalIneq} 
   \max ( 1 , C_{\rm b}(t_0,z_0) )^{N-1}
   \max \left( 1 , \frac{t_0}{\lambda} C_\circ \right)
    e^{\lambda N t_0} 
    < \frac{1}{(1-e^{-\lambda(N-1)t_0})^{1/p}}, 
\end{equation} 
 then we have
\begin{equation}
\nonumber %
  \E\bigl[ |\mathcal{H}(z,t)|^p \bigr] < \infty,
  \qquad (z,t)\in\widebar{\Gamma}_{z_0,t_0}(\C^d\times[0,t_0]).
\end{equation}
\end{proposition}
\begin{proof}
\noindent
For $N\geq 2$, the boundary size of the dominating tree has the probability
generating function %
\begin{equation}\label{eq:PGFboundary}
  \E\bigl[s^{|\widetilde{{\cal K}}_{N,t}^{\rm b} |}\bigr]
     = 
     \frac{s e^{-\lambda t}}{\bigl(1-s^{N-1}
       (1-e^{-\lambda(N-1)t})\bigr)^{1/(N-1)}},
       \qquad 
       0\le s<s_* (t),
\end{equation}
 see \cite[Example 13.2]{harris1963}
 or \cite[Example 5 page 109]{athreya}, 
 with radius of convergence
\begin{equation}\label{eq:RadiusConv}
 s_* (t):=\frac{1}{(1-e^{-\lambda(N-1)t})^{1/(N-1)}}. 
\end{equation}
 From \eqref{eq:Hbound} and \eqref{eq:DomRelation}, we have
\begin{align*} 
  & \E\bigl[ |\mathcal{H}(z,t)|^p \bigr]
  \leq
  \E\left[
    \left(
   \max (1 , C_{\rm b}(t_0,z_0) )^{|\widetilde{{\cal K}}_{N,t}^{\rm b} |} 
   \max \left(1 , \frac{t_0}{\lambda}C_\circ \right)^{|\widetilde{{\cal K}}_{N,t}^\circ |}
    \prod_{\kappa\in\widetilde{{\cal K}}_{N,t} }e^{\lambda T_\kappa }
    \right)^p
\right]
  \\
\nonumber
                        & 
 \quad = e^{-\lambda t p /(N-1)}
  \E\left[
    \left(
    \max (1 , C_{\rm b}(t_0,z_0) )^{|\widetilde{{\cal K}}_{N,t}^{\rm b} |} 
    \max \left(1 , \frac{t_0}{\lambda} C_\circ \right)^{
 ( |\widetilde{{\cal K}}_{N,t}^{\rm b} | -1)/(N-1)
}
    e^{\lambda t N |\widetilde{{\cal K}}_{N,t}^{\rm b} |/(N-1)}
    \right)^p
\right]
  \\
\nonumber
    & \quad \leq
\E\left[
    \left(
    \max (1 , C_{\rm b}(t_0,z_0) )
    \max \left(1 , \frac{t_0}{\lambda} C_\circ \right)^{
  1 /(N-1)
}
    e^{\lambda t N /(N-1)}
    \right)^{p|\widetilde{{\cal K}}_{N,t}^{\rm b} |}
\right]
\\
\nonumber
& \quad < \infty, \quad  (z,t)\in\widebar{\Gamma}_{z_0,t_0}(\C^d\times[0,t_0]),
\quad t>0, 
\end{align*} 
provided that
$$
\left(
    \max ( 1 , C_{\rm b}(t_0,z_0) )  
    \max \left(1 , \frac{t_0}{\lambda} C_\circ\right)^{
 1 /(N-1)
}
    e^{\lambda t N /(N-1)}
    \right)^p
    <
 s_* (t) =\frac{1}{(1-e^{-\lambda(N-1)t})^{1/(N-1)}}, 
$$  
 i.e.
 $$
   \max ( 1 , C_{\rm b}(t_0,z_0) )^{N-1}
   \max \left( 1 , \frac{t_0}{\lambda} C_\circ\right)
    e^{\lambda N t} 
    <
\frac{1}{(1-e^{-\lambda(N-1)t})^{1/p}}. 
$$  
\end{proof} 
\noindent
Increasing the time horizon $t_0$ makes the condition \eqref{eq:CriticalIneq} harder to satisfy, as the left-hand side increases through the exponential factor and the possible $t_0$-dependence of $C_{\rm b}(t_0,z_0)$, whereas the right-hand side decreases to $1$ as $t_0$ grows.
 On the other hand, \eqref{eq:CriticalIneq}
    can be replaced by a simpler and stronger condition
    similarly to \eqref{eq:explicit-threshold-1d-0}-\eqref{eq:explicit-threshold-23d-0} above.

\section{Classical solutions}
\label{s6} 
\noindent
In this section, we show that the probabilistic representation %
\eqref{eq:Estimator} can yield not only a mild but also
a {classical} solution to the Cauchy problem \eqref{eq:WaveMain}-\eqref{eq:PolyNonlin}. 

\begin{assumption}\label{assu2}
  Assume that $\psi\in {\cal C}^2 \bigl(
  \widebar{\Gamma}_{z_0,t_0}(\C^d)
  \bigr)$, and 
\begin{enumerate}[i)] %
\item $\phi \in {\cal C}^2 \bigl(
  \widebar{\Gamma}_{z_0,t_0}(\C^d)
  \bigr)$ for $d=1$, 
\item $\phi\in {\cal C}^3 \bigl(
  \widebar{\Gamma}_{z_0,t_0}(\C^d)
  \bigr)$ for $d=2,3$.
\end{enumerate}
\end{assumption}
\noindent
 Assumption~\ref{assu2} is required
 to ensure sufficient differentiability of
 $\mathcal{H}(z,t,\kappa)$
 in \eqref{eq:GlobalWeight}
 for the existence of classical solutions, 
 as shown in the next result.
\begin{theorem}[Existence of classical solutions]
\label{thm:classical}
 Let $N \geq 1$ and $(z_0,t_0)\in\C^d \times (0,\infty)$. 
 Suppose that Assumption~\ref{assu2} holds, and that
 \eqref{jfklfdsa1a} is satisfied if $N\geq 2$. 
 Then, for every $(z,t)\in\widebar{\Gamma}_{z_0,t_0}(\C^d\times[0,t_0])$
 we have 
$
 \E\bigl[ |\mathcal{H}(z,t)| \bigr] < \infty$,
 and the probabilistic representation %
\[
v(z,t) := \E [\mathcal{H}(z,t)],
\quad (z,t) \in \widebar{\Gamma}_{z_0,t_0}(\C^d\times[0,t_0]), 
\]
 satisfies 
\begin{enumerate}[(i)] %
\item \label{it:C2space}
  $v\in {\cal C}^2_z\bigl(
  \widebar{\Gamma}_{z_0,t_0}(\C^d\times[0,t_0])
  \bigr)$;
\item \label{it:C2time}
  $v\in {\cal C}^2_t\bigl(
  \widebar{\Gamma}_{z_0,t_0}(\C^d\times[0,t_0])
  \bigr)$;
\item \label{it:PDE}
      $\partial_{tt}v-c^2\Delta v=f(v)$ on
      $\widebar{\Gamma}_{z_0,t_0}(\C^d\times[0,t_0])$, $t>0$;
\item \label{it:IC}
  for all
  $z \in \widebar{\Gamma}_{z_0,t_0}(\C^d)$,
  $\displaystyle
  \lim_{(z',t)\to (z,0) } v( z' ,t)=\phi (z)$ and
  $\displaystyle
  \lim_{(z',t)\to (z,0) } \partial_t v( z' ,t)=\psi (z)$,
 where $(z',t)\in\widebar{\Gamma}_{z_0,t_0}(\C^d\times[0,t_0])$,
 $t>0$.
\end{enumerate}
\end{theorem} 

\begin{proof}
\newcommand{\proofpart}[1]{\par\medskip\noindent\textbf{#1}\par\smallskip}

\proofpart{(i) ${\cal C}^2$ regularity of $v$ in space.}
\noindent
\textbf{Step 1 (Pointwise ${\cal C}^2_z$ regularity for $\mathcal{H}(z,t)$).}
In the product expression \eqref{eq:GlobalWeight}, the $z$-dependence
appears only in the boundary factors $W^{\rm b}_{(1)}$, and each such factor
is ${\cal C}^2$ in $z$ by \eqref{eq:Ipartial} and Assumption~\ref{assu2},
 hence $\mathcal{H}(\cdot,t)$ is ${\cal C}^2$ in $z$, %
 $\PP$-{\em a.s.} %

\noindent\textbf{Step 2 (Dominated convergence).}
 In the product expression 
\eqref{eq:GlobalWeight}, the variable $z$ appears only through the boundary
terms $W^{\rm b} (z,t,\kappa )$ in \eqref{eq:Ipartial}, which are
 ${\cal C}^2$ in $z$ by Assumption~\ref{assu2}. 
Hence, the second derivative $\partial_{z_kz_l}{\cal H}$,
$1\le k,l\le d$, contains at most $|\widetilde{{\cal K}}_{N,t}^{\rm b} |^2$
 summands, up to a multiplicative constant. 
 In addition, similarly to the derivation of the bound \eqref{eq:Hbound},
using the cone bounds \eqref{eq:CD1-Cone} and \eqref{eq:CDhigh-Cone}
together with \eqref{eq:Ipartial}-\eqref{eq:Io} and
 \eqref{eq:Apartial0}-\eqref{eq:Apartial}, there exists
 a constant $C=C(\phi,\psi,z_0,t_0,d)>0$ such that
\begin{align*} 
   \sum_{1\le k,l\le d}\bigl|\partial_{z_k z_l}\mathcal{H}(z,t)\bigr|
    \le 
  C\,|\widetilde{{\cal K}}_{N,t}^{\rm b} |^2
  \max (1 , C_{\rm b}(t_0,z_0) )^{|\widetilde{{\cal K}}_{N,t}^{\rm b} |}
  \max \left(1 , \frac{t_0}{\lambda}C_\circ\right)^{|\widetilde{{\cal K}}_{N,t}^\circ |}
  \prod_{\kappa\in\widetilde{{\cal K}}_{N,t} }e^{\lambda T_\kappa }, 
\end{align*} 
 $(z,t)\in\widebar{\Gamma}_{z_0,t_0}(\C^d\times[0,t_0])$. 
 For $N=1$, we have
   $|\widetilde{{\cal K}}_{N,t}^{\rm b} |=1$,
 $|\widetilde{{\cal K}}_{N,t}^\circ|\sim\mathrm{Poi}(\lambda t)$, and
 \begin{align*} 
   \sum_{1\le k,l\le d}\bigl|\partial_{z_k z_l}\mathcal{H}(z,t)\bigr|
   \leq 
  C 
  \max (1 , C_{\rm b}(t_0,z_0) )
  \max \left(1 , \frac{t_0}{\lambda}C_\circ\right)^{|\widetilde{{\cal K}}_{N,t}^\circ |}
 e^{\lambda t_0}, 
\end{align*} 
 so that
  finiteness of expectations follows from \eqref{eq:PoissonPGF}.
  For $N\ge2$, 
  using \eqref{eq:DomRelation} we have  
\begin{align*} 
 &  \sum_{1\le k,l\le d}\bigl|\partial_{z_k z_l}\mathcal{H}(z,t)\bigr|
  \\
  & \qquad  
   \le 
  C\,|\widetilde{{\cal K}}_{N,t}^{\rm b} |^2
  \max (1 , C_{\rm b}(t_0,z_0) )^{|\widetilde{{\cal K}}_{N,t}^{\rm b} |}
  \max \left(1 , \frac{t_0}{\lambda}C_\circ\right)^{|\widetilde{{\cal K}}_{N,t}^{\rm b}|/(N - 1)}
  e^{\lambda t_0
   N |\widetilde{{\cal K}}_{N,t}^{\rm b}|/(N - 1)}
, 
\end{align*} 
 $(z,t)\in\widebar{\Gamma}_{z_0,t_0}(\C^d\times[0,t_0])$.
 Next, differentiating
  the PGF
  $G_t(s):=\E \big[ s^{|\widetilde{{\cal K}}_{N,t}^{\rm b} |} \big]$ 
  in \eqref{eq:PGFboundary}
  with respect to $s$ gives 
\begin{align} 
  \label{fklsfa}
  \E\!\big[|\widetilde{{\cal K}}_{N,t}^{\rm b} |^2 s^{|\widetilde{{\cal K}}_{N,t}^{\rm b} |}\big]
& = s^2 G_t''(s) + s G_t'(s)
\\
\nonumber
& = e^{-\lambda t}\big(
  s\,B^{-\alpha-1}
+ (N+s)\,r\,s^{N}\,B^{-\alpha-1}
+ N\,r^2 \,s^{2N-1}\,B^{-\alpha-2}
\big),
\end{align}
 with the abbreviations
 $r:=1-e^{-\lambda(N-1)t}$, $B:=1-r s^{N-1}$, $\alpha:=1/(N-1)$,
 and the same radius of convergence $s_* (t)$
 in \eqref{eq:RadiusConv} as $G_t(s)$. 
 Choosing
\[
s : = \max (1 , C_{\rm b}(t_0,z_0) )
      \max \left(1 , \frac{t_0}{\lambda} C_\circ\right)^{1/(N-1)}
        e^{\lambda t_0 N/(N-1)},
\]
 in \eqref{fklsfa}, the bound
  \eqref{jfklfdsa1a}, i.e. \eqref{eq:CriticalIneq} applied with $p=1$,
  yields $s<s_* (t_0)\leq s_* (t)$ when $t\le t_0$, 
  and therefore the integrability of
  $$
  |\widetilde{{\cal K}}_{N,t}^{\rm b} |^2
  \max (1 , C_{\rm b}(t_0,z_0) )^{|\widetilde{{\cal K}}_{N,t}^{\rm b} |}
  \max \left(1 , \frac{t_0}{\lambda}C_\circ\right)^{|\widetilde{{\cal K}}_{N,t}^{\rm b}|/(N - 1)}
  e^{\lambda t_0
   N |\widetilde{{\cal K}}_{N,t}^{\rm b}|/(N - 1)}
  $$
 for all $t\in [0,t_0]$.
 Hence, by the differentiability lemma, see, e.g.
 Theorem~12.5 in \cite{Schilling2017}, we obtain 
\[
  \partial_{z_k z_l}v(z,t)=\E\!\left[\partial_{z_k z_l}\mathcal{H}(z,t)\right],
  \quad
   k,l = 1,\ldots , d, 
\]
 which proves \eqref{it:C2space}. 

\proofpart{(ii) ${\cal C}^2$ regularity of $v(t,z)$ in time.}
\noindent
By Theorem~\ref{thm:mild-existence}, $v$ satisfies the mild identities
\eqref{eq:Mild1D} for $d=1$, resp.\ \eqref{eq:Mild2Dpolar}-\eqref{eq:Mild3Dpolar} for $d=2,3$. The terms involving the initial data are ${\cal C}^2$ in $t$ (see Section~\ref{subsec:LinearTheory} and \cite[§2.4, Theorems~1-3]{Eva10}). It remains to treat the nonlinear Duhamel term.
 Writing these terms on a fixed domain
 using the change of variables $y\mapsto (t-s)y$ gives:
\[
\begin{cases}
\displaystyle
\frac{1}{2}\int_0^t (t-s)\!\int_{-1}^1 f (
v(z+c(t-s)y,s) ) \,dy\,ds, & d=1,\\[12pt]
\displaystyle
\frac{1}{2\pi}\int_0^t (t-s)\!\int_{B_2(0,1)} 
\frac{f(v(z+c(t-s)y,s))}{\sqrt{1-|y|^2}}\,dy\,ds, & d=2,\\[12pt]
\displaystyle
\frac{1}{4\pi}\int_0^t (t-s)\!\int_{S_2 (0,1)}
f(v(z+c(t-s)y,s))\,\sigma^{(2)}_1 (dy)\,ds, & d=3.
\end{cases}
\]
Differentiating in $t$ brings two types of contributions: explicit derivatives of $(t-s)$ and chain-rule derivatives through $z+c(t-s) y $.
The latter use spatial derivatives of $v$. Since \eqref{it:C2space} already gives $v\in {\cal C}^2_z$ on the cone and $f$ is a polynomial, the integrands of the first and second $t$-derivatives are bounded on the fixed domains by functions integrable in $s$. %
 Hence the dominated convergence theorem legitimizes differentiating under the integral twice for $t$, yielding $v\in {\cal C}^2_t$ on $\widebar{\Gamma}_{z_0,t_0}(\C^d\times[0,t_0])$, and proving~\eqref{it:C2time}.

\proofpart{(iii) PDE identity and (iv) Initial conditions.}
\noindent
 We note that $v$ is ${\cal C}^2$ in $z$ by \eqref{it:C2space} and in $t$ by \eqref{it:C2time}, hence 
in Section~\ref{subsec:LinearTheory}
the source term
 $g = f \circ v$
 is ${\cal C}^2$. 
 In addition,
 by Theorem~\ref{thm:mild-existence}, 
 $v$ satisfies the
 mild identities
\eqref{eq:Mild1D} for $d=1$, resp.\ \eqref{eq:Mild2Dpolar}-\eqref{eq:Mild3Dpolar} for $d=2,3$. 
 Hence, 
 \eqref{it:PDE} and 
 \eqref{it:IC} hold 
 from
 \cite[§2.4, Theorems~1~and~4]{Eva10},
 \cite[§2.4, Theorems~3~and~4]{Eva10},
 \cite[§2.4, Theorems~2~and~4]{Eva10},
 respectively for $d=1,2,3$.
\end{proof}
\noindent

\section{Numerical experiments} 
\label{s7} 
\noindent
 All Monte Carlo experiments in this section
  are using $10^7$ random
  samples.
\subsection{Nonlinear Klein-Gordon equation} %
\noindent
 Letting $c:=1$, we consider the defocusing cubic Klein-Gordon equation
 \begin{equation}
   \label{ddefocusing} 
    \partial_{tt}u - \Delta u + u + u^3 = 0 
\end{equation} 
 see, e.g., \cite{KillipStovallVisan2012},
 with sine initial condition 
\[
    \phi(z)=\sin(\pi z_1)\sin(\pi z_2), \quad
    \psi(z)=-\sin(\pi z_1)\sin(\pi z_2)  
\] 
in dimension $d=2$. 
Figures~\ref{fig:Stability-1} and \ref{fig:Stability-1-2}
plot the Monte Carlo estimate \eqref{eq:Estimator}
 of the solution of \eqref{ddefocusing} 
 on $[0,1]\times [0,1]$
 and the difference 
 between Monte Carlo and finite differences 
 estimates
 obtained using the command {\tt NDSolveValue} in Mathematica. 

\vspace{-0.6cm}

\begin{figure}[H]
\centering %
\begin{subfigure}{0.45\textwidth}
  \includegraphics[width=\textwidth]{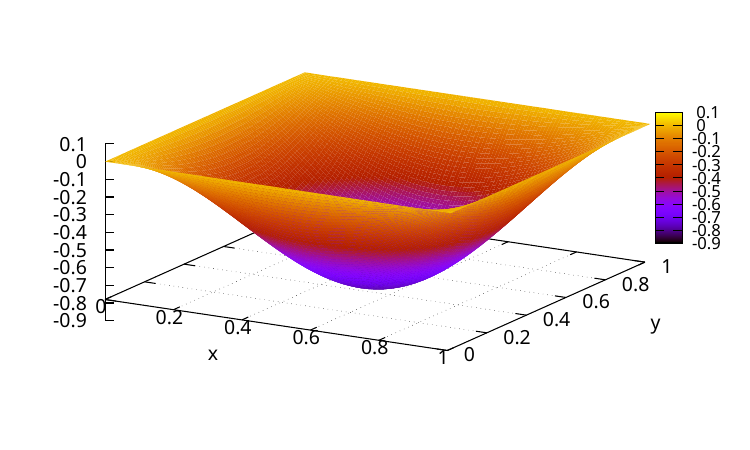}
  \vspace{-1.2cm}
\subcaption{Monte Carlo estimate \eqref{eq:Estimator}.} 
\label{fig:Stability-1}
\end{subfigure}
\hskip0.5cm
\begin{subfigure}{0.45\textwidth}
  \includegraphics[width=\textwidth]{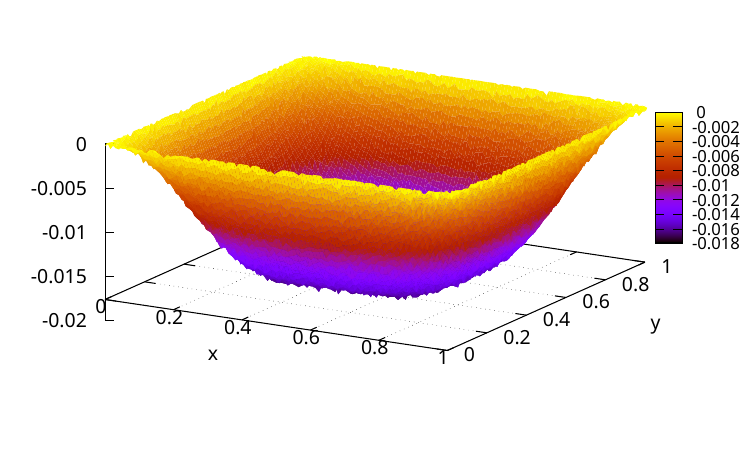}
  \vspace{-1.2cm}
\subcaption{Discrepancy with finite differences.}
\label{fig:Stability-1-2}
\end{subfigure}
\caption{Numerical solutions of \eqref{ddefocusing} with $t=0.5$ and $d=2$.} 
\end{figure}

\vspace{-0.4cm}

\subsection{Defocusing elliptic problem} 
\noindent 
Letting $c:=\Im$, we consider the nonlinear
elliptic equation 
\begin{equation}
\label{djklsd1} 
 \partial_{tt}u+\Delta u + u - u^3 = 0, 
\end{equation}
which can be obtained from the focusing
cubic Klein-Gordon equation
\begin{equation}
\nonumber 
    \partial_{tt}u - \Delta u + u - u^3 = 0 
\end{equation} 
 see e.g. \cite{NakanishiSchlag2011}, 
 by the complex transformation $z\mapsto i z$.
\subsubsection*{Travelling-wave initial data}
\noindent 
 We consider the equation \eqref{djklsd1} with the initial data and
 corresponding traveling wave solutions listed in Table~\ref{tab:IC}.

 \begin{table}[H]
\centering\small
\begin{tabular}{|c|c|c|c|}
  \hline
  \rule{0pt}{10pt}$d$ & $\phi(z)$ & $\psi(z)$ & Closed-form solution \\[1pt]
  \hline
  \rule{0pt}{20pt}1 &
    $\displaystyle\tanh \Bigl(\frac{i}{\sqrt6} z\Bigr)$ &
    $\displaystyle-\sqrt{\frac23} 
    \text{sech}^2 \Bigl(\frac{i}{\sqrt6} z\Bigr)$ &
    $\displaystyle\tanh \Bigl(\frac{iz-2t}{\sqrt6} \Bigr)$ \\[10pt]
  \hline
  \rule{0pt}{20pt}2 &
    $\displaystyle\tanh \Bigl(i\frac{z_1+z_2}{2\sqrt3}\Bigr)$ &
    $\displaystyle-\sqrt{\frac23} 
    \text{sech}^2 \Bigl(i\frac{z_1+z_2}{2\sqrt3}\Bigr)$ &
    $\displaystyle\tanh \Bigl(\frac{1}{\sqrt6} \Bigl(i\frac{z_1+z_2}{\sqrt2}-2t\Bigr)\Bigr)$ \\[10pt]
  \hline
  \rule{0pt}{20pt}3 &
    ~$\displaystyle\tanh \Bigl(i\frac{z_1+z_2+z_3}{3\sqrt2}\Bigr)$~ &
    ~$\displaystyle-\sqrt{\frac23} 
    \text{sech}^2 \Bigl(i\frac{z_1+z_2+z_3}{3\sqrt2}\Bigr)$~ &
    ~$\displaystyle\tanh \Bigl(\frac{1}{\sqrt6} \Bigl(i\frac{z_1+z_2+z_3}{\sqrt3}-2t\Bigr)\Bigr)$~ \\[10pt]
  \hline
\end{tabular}
\caption{Closed-form solutions of \eqref{djklsd1}.}
\label{tab:IC}
\end{table}

\vskip-0.6cm

\noindent 
 Figures~\ref{f2-1-0}-\ref{f2-3-0}
 are plotted as functions of time $t\in [0,1.8]$
 for $z=(-1,\ldots ,-1)$,
 and show the Monte Carlo estimates in addition
 to the explosion of standard deviations (SD).

\vskip-0.1cm

\begin{figure}[H]
\centering %
\begin{subfigure}{0.45\textwidth}
    \includegraphics[width=1.\textwidth]
      {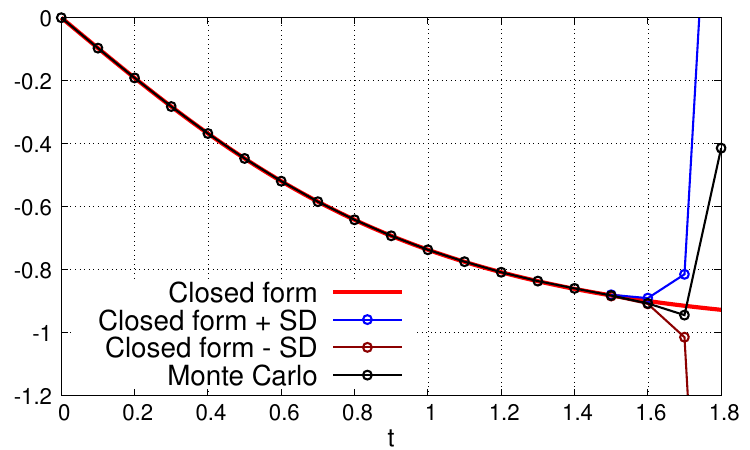} %
\subcaption{Real part.} 
\end{subfigure}
\begin{subfigure}{0.45\textwidth}
    \includegraphics[width=1.\textwidth]
      {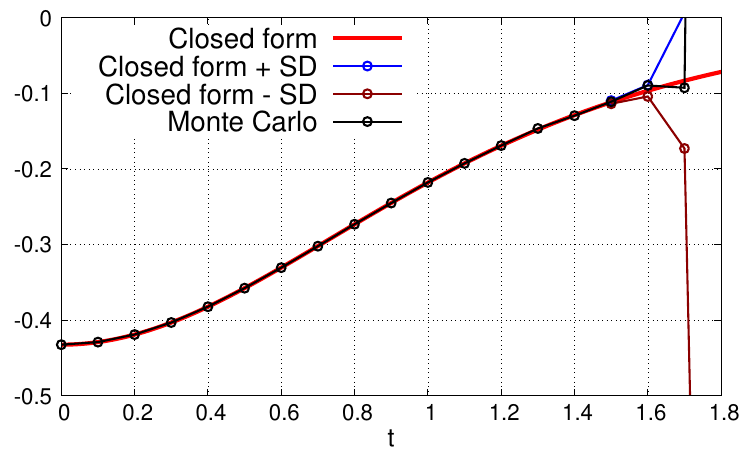} %
\subcaption{Imaginary part.} 
\end{subfigure}
\caption{Monte Carlo estimate \eqref{eq:Estimator} in dimension $d=1$.} 
\label{f2-1-0} 
\end{figure}

\begin{figure}[H]
\centering %
\begin{subfigure}{0.45\textwidth}
         \includegraphics[width=1.\textwidth]
      {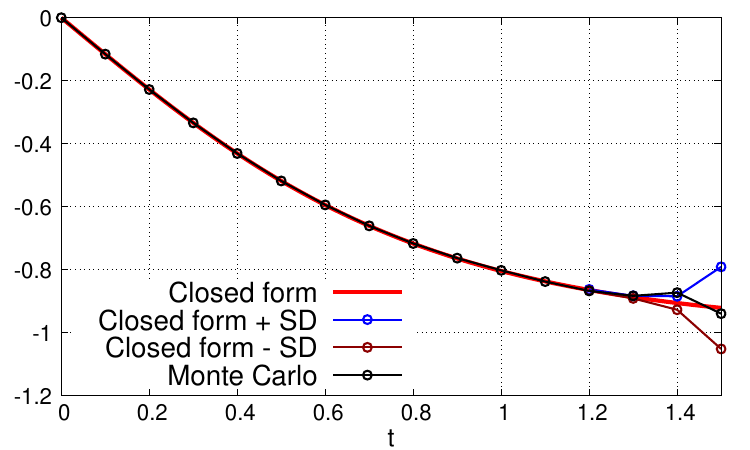} %
\subcaption{Real part.} 
\end{subfigure}
\begin{subfigure}{0.45\textwidth}
    \includegraphics[width=1.\textwidth]
      {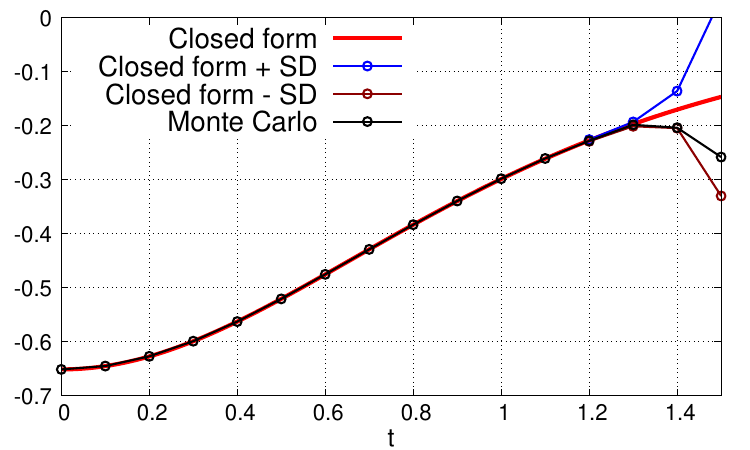} %
\subcaption{Imaginary part.} 
\end{subfigure}
\caption{Monte Carlo estimate \eqref{eq:Estimator} in dimension $d=2$.} 
\end{figure}

\begin{figure}[H]
\centering %
\begin{subfigure}{0.45\textwidth}
         \includegraphics[width=1.\textwidth]
      {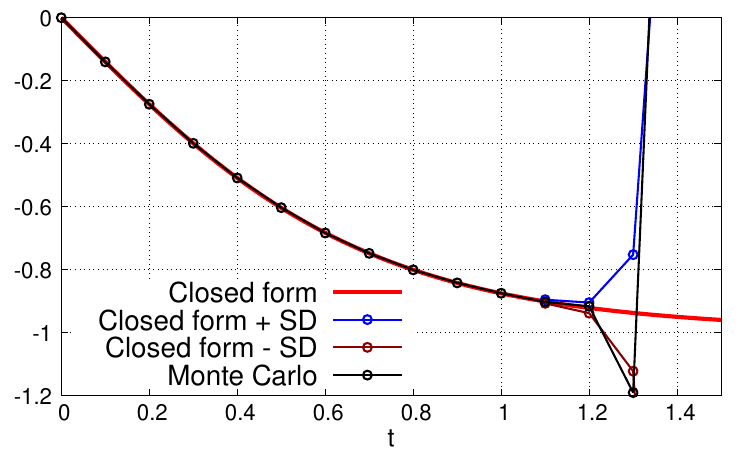} %
\subcaption{Real part.} 
\end{subfigure}
\begin{subfigure}{0.45\textwidth}
    \includegraphics[width=1.\textwidth]
      {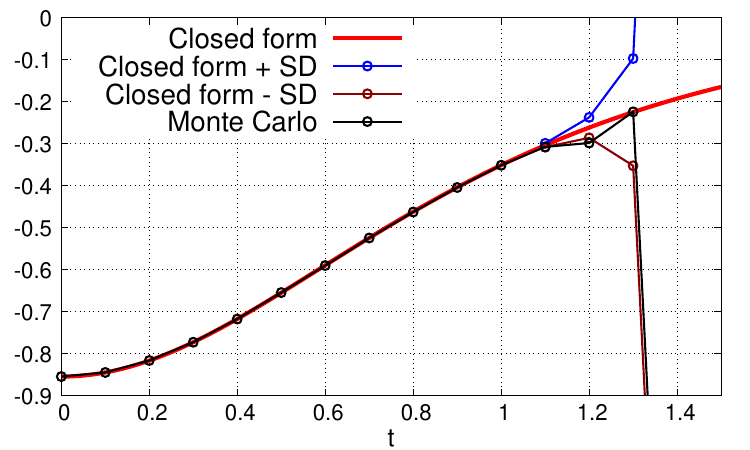} %
\subcaption{Imaginary part.} 
\end{subfigure}
\caption{Monte Carlo estimate \eqref{eq:Estimator} in dimension $d=3$.} 
\label{f2-3-0} 
\end{figure}

\vspace{-0.4cm}

\noindent
Figure~\ref{fig:runtimes} presents a comparison of runtimes in seconds for Figures~\ref{f2-1-0}-\ref{f2-3-0}.

\begin{figure}[H]
\centering %
\includegraphics[width=0.9\textwidth]
      {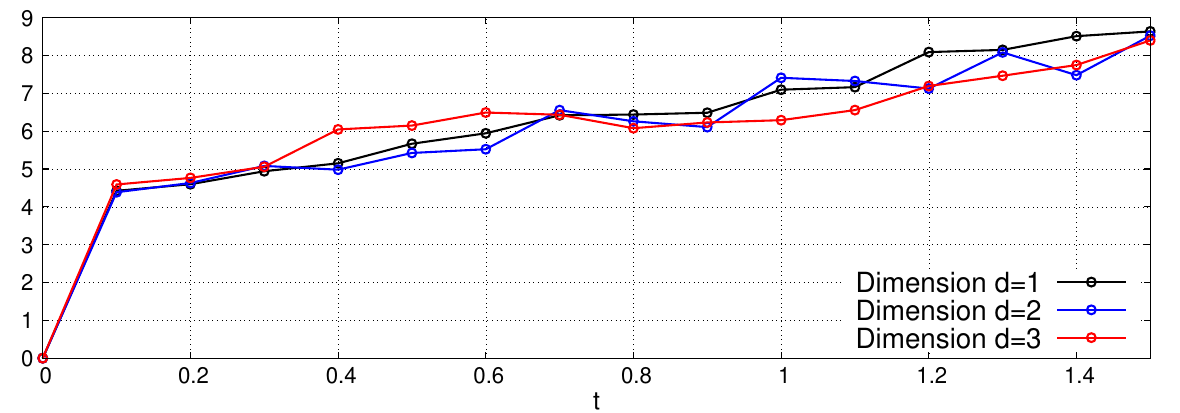}
\vspace{-0.2cm}
\caption{Comparison of runtimes (s) for Figures~\ref{f2-1-0}-\ref{f2-3-0}.} 
\label{fig:runtimes}
\end{figure}

\vspace{-0.4cm}

\noindent
 Computer codes and numerical results for
 Figures~\ref{f2-1-0}-\ref{f2-3-0} are available at\\
\centerline{\url{https://github.com/chanjuanyin/Nonlinear_Wave_simulations}.} 

\subsubsection*{Sine initial data} 

\noindent
We consider the elliptic equation \eqref{djklsd1} in dimension $d=1$
with the sine initial data
$$
\phi (x ) = \sin (\pi x) \quad \mbox{and}
\quad \psi (x) = - \sin (\pi x), \quad x\in [0,1].
$$
Figure~\ref{fig:Stability-1-2-0} 
compares Monte Carlo and explicit finite difference estimates of 
$u(x,t)$ with $x\in [0,1]$ and $t\in [0,0.4]$ 
using the command {\tt NDSolveValue} in Mathematica.
 We observe that the Monte Carlo and finite differences match until time $t=0.35$, beyond which the finite difference scheme becomes unstable.

\begin{figure}[H]
\centering %
\includegraphics[width=0.9\textwidth]
      {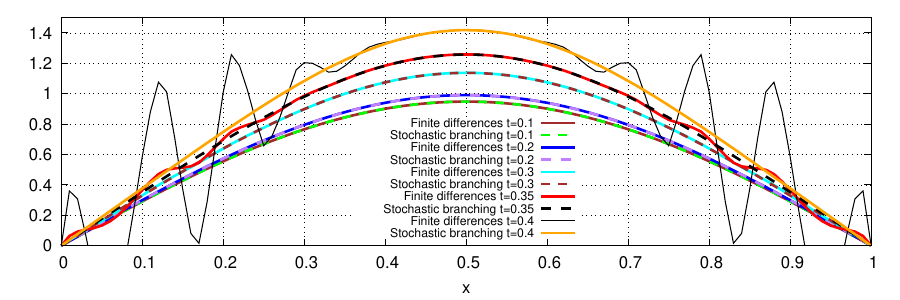}
\vspace{-0.4cm}
\caption{Comparisons with explicit finite differences.} 
\label{fig:Stability-1-2-0}
\end{figure}

\vspace{-0.4cm}

\noindent 
Figure~\ref{fig:Stability-2-2-0} 
compares Monte Carlo and implicit finite difference estimates of 
$u(x,t)$ with $x\in [0,1]$ and $t\in [0,0.7]$
using the command {\tt NDSolveValue} 
with the {\tt LinearlyImplicitEuler} option 
in Mathematica. 

\begin{figure}[H]
\centering %
\includegraphics[width=0.9\textwidth]
      {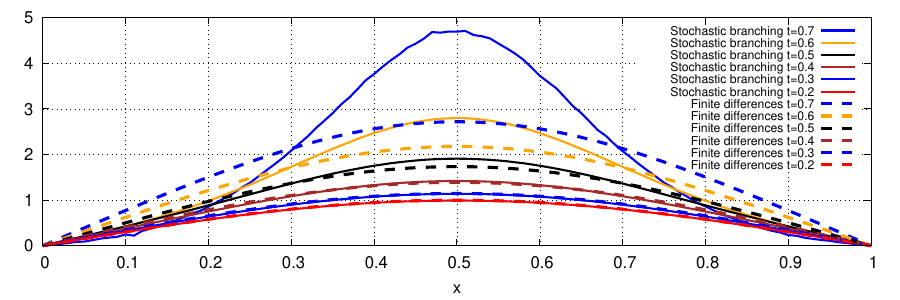}
\vspace{-0.4cm}
\caption{Comparisons with implicit finite differences.} 
\label{fig:Stability-2-2-0}
\end{figure}

\vspace{-0.4cm}

\noindent
As expected, the implicit scheme is more stable but exhibits
a loss of accuracy compared to the explicit scheme due
to loss of energy conservation. 

\subsection{Focusing elliptic problem} 
\noindent
 We consider the 
 focusing equation 
 \begin{equation}
   \label{focusing} 
    \partial_{tt}u + \Delta u + u + u^3 = 0, 
\end{equation}
  with $(x,y)\in [0,1]\times [0,1]$ and $c=\Im$,
   which
 can be obtained from the defocusing
 cubic Klein-Gordon equation \eqref{ddefocusing} 
 by the complex transformation $z\mapsto i z$,
 and can be viewed as an ill-posed problem. 
 Figures~\ref{fig:Stability-2} and \ref{fig:Stability-2-2}
 compare the Monte Carlo method to the finite differences estimates of the solution
$u(x,y,0.5)$ obtained using the command {\tt NDSolveValue} in Mathematica.  
 We note that the finite difference method is clearly
 unstable for this type of problem.
 
\vspace{-0.6cm}

\begin{figure}[H]
\centering %
\begin{subfigure}{0.49\textwidth}
\includegraphics[width=\textwidth]{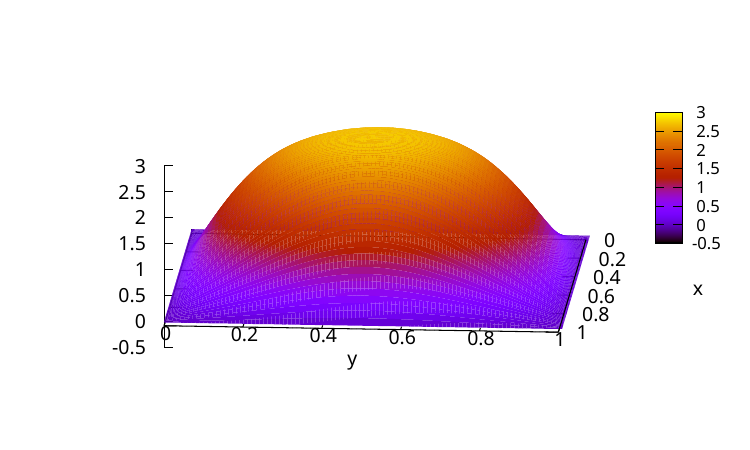}
\vspace{-1.2cm}
\subcaption{Monte Carlo estimate \eqref{eq:Estimator}.} %
  \label{fig:Stability-2}
\end{subfigure}
\begin{subfigure}{0.49\textwidth}
\includegraphics[width=\textwidth]{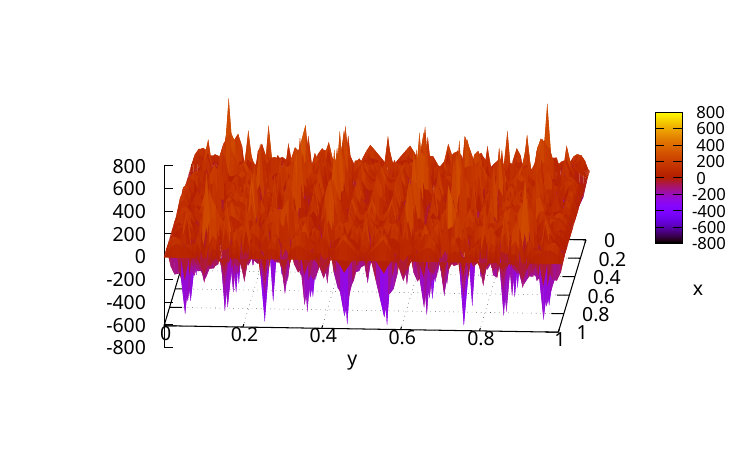}
\vspace{-1.2cm}
\subcaption{Finite differences estimation.}
\label{fig:Stability-2-2}
\end{subfigure}
\caption{Numerical solutions of \eqref{focusing} with $t=0.5$ and $d=2$.} 
\end{figure}

\vspace{-0.4cm}

\noindent
Additional numerical experiments confirming the above observations
are available at

\centerline{\url{https://github.com/nprivaul/wave_equation}} 

\noindent
using the deep Galerkin and finite difference methods
in Python. 

\paragraph{Acknowledgement.}
 We thank Qiao Huang (Southern University of Science and Technology, China) 
 for useful suggestions.

\footnotesize

\def\cprime{$'$} \def\polhk#1{\setbox0=\hbox{#1}{\ooalign{\hidewidth
  \lower1.5ex\hbox{`}\hidewidth\crcr\unhbox0}}}
  \def\polhk#1{\setbox0=\hbox{#1}{\ooalign{\hidewidth
  \lower1.5ex\hbox{`}\hidewidth\crcr\unhbox0}}} \def\cprime{$'$}


\begin{thebibliography}{DMTW19}

\bibitem[AN72]{athreya}
K.B. Athreya and P.E. Ney.
\newblock {\em Branching processes}, volume 196 of {\em Die Grundlehren der
  mathematischen Wissenschaften}.
\newblock Springer-Verlag, New York-Heidelberg, 1972.

\bibitem[AT92]{AgemiTakamura1992}
R.~Agemi and H.~Takamura.
\newblock The lifespan of classical solutions to nonlinear wave equations in
  two space dimensions.
\newblock {\em Hokkaido Math. J.}, 21:517--542, 1992.

\bibitem[BM10]{bakhtin}
Y.~Bakhtin and C.~Mueller.
\newblock Solutions of semilinear wave equation via stochastic cascades.
\newblock {\em Commun. Stoch. Anal.}, 4(3):425--431, 2010.

\bibitem[CDM06]{cartier}
P.~Cartier and C.~DeWitt-Morette.
\newblock {\em Functional integration: action and symmetries.}
\newblock Camb. Monogr. Math. Phys. Cambridge: Cambridge University Press,
  2006.

\bibitem[CLM08]{chakraborty}
S.~Chakraborty and J.A. L\'{o}pez-Mimbela.
\newblock Nonexplosion of a class of semilinear equations via branching
  particle representations.
\newblock {\em Adv. Appl. Probab.}, 40:250--272, 2008.

\bibitem[DMT08]{dalang}
R.C. Dalang, C.~Mueller, and R.~Tribe.
\newblock A {F}eynman--{K}ac-type formula for the deterministic and stochastic
  wave equations and other {P}.{D}.{E}.'s.
\newblock {\em Trans. Amer. Math. Soc.}, 360(9):4681--4703, 2008.

\bibitem[DMTW19]{waymire}
R.~Dascaliuc, N.~Michalowski, E.~Thomann, and E.C. Waymire.
\newblock Complex {B}urgers equation: a probabilistic perspective.
\newblock In {\em Sojourns in probability theory and statistical physics. {I}.
  {S}pin glasses and statistical mechanics, a {F}estschrift for {C}harles {M}.
  {N}ewman}, volume 298 of {\em Springer Proc. Math. Stat.}, pages 138--170.
  Springer, Singapore, 2019.

\bibitem[Eva10]{Eva10}
L.~C. Evans.
\newblock {\em Partial Differential Equations}, volume~19 of {\em Grad. Stud.
  Math.}
\newblock Amer. Math. Soc., Providence, RI, 2nd edition, 2010.

\bibitem[Gol51]{goldstein1951random}
S.~Goldstein.
\newblock On diffusion by discontinuous movements and on the telegraph
  equation.
\newblock {\em Q. J. Mech. Appl. Math.}, 4:129--156, 1951.

\bibitem[Har63]{harris1963}
T.E. Harris.
\newblock {\em The theory of branching processes}, volume 119 of {\em Die
  Grundlehren der mathematischen Wissenschaften}.
\newblock Springer-Verlag, Berlin; Prentice Hall, Inc., Englewood Cliffs, NJ,
  1963.

\bibitem[HLT21]{labordere2}
P.~Henry-Labord\`ere and N.~Touzi.
\newblock Branching diffusion representation for nonlinear {C}auchy problems
  and {M}onte {C}arlo approximation.
\newblock {\em Ann. Appl. Probab.}, 31(5):2350--2375, 2021.

\bibitem[HLTT14]{laborderespa}
P.~Henry-Labord\`ere, X.~Tan, and N.~Touzi.
\newblock A numerical algorithm for a class of {BSDE}s via the branching
  process.
\newblock {\em Stochastic Process. Appl.}, 124(2):1112--1140, 2014.

\bibitem[INW69]{inw}
N.~Ikeda, M.~Nagasawa, and S.~Watanabe.
\newblock Branching {M}arkov processes {I}, {II}, {III}.
\newblock {\em J. Math. Kyoto Univ.}, 8-9:233--278, 365--410, 95--160,
  1968-1969.

\bibitem[Kac74]{kac1974stochastic}
M.~Kac.
\newblock A stochastic model related to the telegrapher's equation.
\newblock {\em Rocky Mt. J. Math.}, 4:497--509, 1974.

\bibitem[KSV12]{KillipStovallVisan2012}
R.~Killip, B.~Stovall, and M.~Visan.
\newblock Scattering for the cubic {K}lein--{G}ordon equation in two space
  dimensions.
\newblock {\em Trans. Amer. Math. Soc.}, 364(3):1571--1631, 2012.

\bibitem[LS97]{sznitman}
Y.~{Le Jan} and A.~S. Sznitman.
\newblock Stochastic cascades and $3$-dimensional {N}avier-{S}tokes equations.
\newblock {\em Probab. Theory Relat. Fields}, 109(3):343--366, 1997.

\bibitem[LZ93]{zhouyi1993}
T.T. Li and Y.~Zhou.
\newblock Life-span of classical solutions to nonlinear wave equations in
  two-space-dimensions. {II}.
\newblock {\em J. Partial Diff. Eq.}, 6(1):17--38, 1993.

\bibitem[McK75]{hpmckean}
H.P. McKean.
\newblock Application of {B}rownian motion to the equation of
  {K}olmogorov-{P}etrovskii-{P}iskunov.
\newblock {\em Comm. Pure Appl. Math.}, 28(3):323--331, 1975.

\bibitem[NPP23]{penent2022fully}
J.Y. Nguwi, G.~Penent, and N.~Privault.
\newblock A fully nonlinear {F}eynman--{K}ac formula with derivatives of
  arbitrary orders.
\newblock {\em J. Evol. Eq.}, 23:Paper No. 22, 29pp., 2023.

\bibitem[NS11]{NakanishiSchlag2011}
K.~Nakanishi and W.~Schlag.
\newblock Global dynamics above the ground state energy for the focusing
  nonlinear {K}lein--{G}ordon equation.
\newblock {\em J. Differential Equations}, 250(5):2299--2333, 2011.

\bibitem[Ram06]{ramirez}
J.M. Ramirez.
\newblock Multiplicative cascades applied to {PDE}s (two numerical examples).
\newblock {\em J. Comput. Phys.}, 214(1):122--136, 2006.

\bibitem[Sch17]{Schilling2017}
R.L. Schilling.
\newblock {\em Measures, Integrals and Martingales}.
\newblock Cambridge University Press, 2nd edition, 2017.

\bibitem[Sko64]{skorohodbranching}
A.V. Skorokhod.
\newblock Branching diffusion processes.
\newblock {\em Teor. Verojatnost. i. Primenen.}, 9:492--497, 1964.

\bibitem[TW11]{TakamuraWakasa2011}
H.~Takamura and K.~Wakasa.
\newblock The sharp upper bound of the lifespan of solutions to critical
  semilinear wave equations in high dimensions.
\newblock {\em J. Differential Equations}, 251(4--5):1157--1171, 2011.

\bibitem[Zho92]{zhouyi1992}
Y.~Zhou.
\newblock Blow up of classical solutions to {$\square u=|u|^{1+\alpha}$} in
  three space dimensions.
\newblock {\em J. Partial Diff. Eq.}, 5(3):21--32, 1992.

\end{thebibliography}
\end{document}